\documentclass[11pt]{article}
\usepackage{amsmath,amssymb}
\usepackage{prettyref}
\usepackage{qr}
\renewcommand{\gets}{\leftarrow}
\renewcommand{\bC}{\mathbb{C}}
\renewcommand{\bR}{\mathbb{R}}
\usepackage{boxedminipage}
\usepackage{enumitem}
\usepackage{booktabs}
\usepackage{threeparttable}

\newcommand{\muqr}{\nu_{\mathsf{IQR}}}

\newcommand{\croot}{c_{\mathsf{root}}}
\newcommand{\Croot}{C_{\mathsf{root}}}

\newcommand{\cd}{C_{\mathsf{D}}}

\newcommand{\false}{\texttt{false}}
\newcommand{\true}{\texttt{true}}

\newcommand{\deflate}{\mathsf{Deflate}}

\newcommand{\scale}{\Sigma}

\newcommand{\chR}{\check{\calR}}
\newcommand{\chr}{\check{r}}

\newcommand{\chs}{\check{s}}
\newcommand{\chq}{\check{q}}

\newcommand{\cc}{c}

\newcommand{\iqr}{\mathsf{IQR}}

\newcommand{\exc}{\mathsf{Exc}}

\newcommand{\find}{\mathsf{Find}}

\newcommand{\comptau}[1]{\mathsf{Tau}^{#1}}

\newcommand{\unif}{\mathrm{Unif}}

\newcommand{\disk}{D}

\newcommand{\tol}{\eta_1}
\newcommand{\pretol}{\eta_2}

\newcommand{\forward}{\beta}

\newcommand{\Rho}{\mathrm{P}}

\newcommand{\Spec}{\mathrm{Spec}}
\newcommand{\dist}{\mathrm{dist}}

\newcommand{\nn}{\ax{\tau^k}}

\newcommand{\exactqr}{\mathrm{iqr}}
\newcommand{\qr}{\mathrm{qr}}

\newcommand{\K}{B}

\newcommand{\shqr}{\mathsf{ShiftedQR}}

\newcommand{\w}{w}

\newcommand{\corner}[2]{#1_{(#2)}}

\newcommand{\E}{\mathbb{E}}
\renewcommand{\P}{\mathbb{P}}
\newcommand{\ax}[1]{\widetilde{#1}}

\newcommand{\pot}{\psi_k}

\renewcommand{\next}[1]{\widehat{#1}}

\newcommand{\fl}{\mathsf{fl}}

\newcommand{\mach}{\textbf{\textup{u}}}
\newcommand{\acc}{\delta}

\newcommand{\wacc}{\omega}

\newcommand{\gap}{\mathrm{gap}}

\renewcommand{\r}{\theta}

\newcommand{\regularize}{\mathsf{RitzOrDecouple}}
\newcommand{\decouple}{{\tt{dec}}}

\newcommand{\Sh}{\mathsf{Sh}}

\newcommand{\optimal}{\mathsf{Optimal}}
\newcommand{\optflag}{{\tt{opt}}}

\DeclareMathOperator{\Dist}{dist}

\newcommand{\cp}{\alpha}
\newcommand{\smalleig}{\mathsf{SmallEig}}

\newcommand{\giv}{\mathsf{giv}}

\newcommand{\rt}{\mathsf{root}}
\newcommand{\gapbound}{\Gamma}

\title{Global Convergence of  Hessenberg Shifted QR II: Numerical Stability}
\author{Jess Banks\thanks{\texttt{jess.m.banks@berkeley.edu}. Supported by NSF GRFP Grant DGE-1752814 and NSF Grant  CCF-2009011.}\\ UC Berkeley \and  Jorge Garza-Vargas\thanks{\texttt{jgarzavargas@berkeley.edu}. Supported by NSF Grant  CCF-2009011.}\\ UC Berkeley \and  Nikhil Srivastava\thanks{\texttt{nikhil@math.berkeley.edu}. Supported by NSF Grant  CCF-2009011.}  \\    UC Berkeley }
\date{\today}
\date{\today}

\begin{document}
\maketitle

\begin{abstract}
    We develop a framework for proving rapid convergence of shifted QR algorithms which use  Ritz values as shifts, in finite arithmetic. Our key contribution is a dichotomy result which addresses the known forward-instability issues surrounding the shifted QR iteration \cite{parlett1993forward}: we give a procedure which provably {\em either} computes a set of approximate Ritz values of a Hessenberg matrix with good forward stability properties, {\em or} leads to early decoupling of the matrix via a small number of QR steps.
    
    Using this framework, we show that the shifting strategy of \cite{banks2021global} converges rapidly in finite arithmetic with a polylogarithmic bound on the number of bits of precision required, when invoked on matrices of controlled eigenvector condition number and minimum eigenvalue gap. 
    
\end{abstract}

\tableofcontents

\section{Introduction}
\newcommand{\shkb}{\mathsf{Sh}_{k,B}}
In Part I of this series \cite{banks2021global} we gave a family of shifting strategies $\shkb$ for which the Hessenberg shifted QR algorithm converges globally and rapidly on nonsymmetric matrices whose eigenvector condition number is bounded, {\em in exact arithmetic}. In this sequel, we show that both the correctness and rapid convergence of these strategies continue to hold in finite (floating point) arithmetic with an appropriate implementation, and prove a bound on the number of bits of precision needed, for matrices with controlled eigenvector condition number and minimum eigenvalue gap. 

To do so, we develop some general tools enabling rigorous finite arithmetic analysis of the shifted QR iteration with any shifting strategy which uses Ritz values as shifts, of which $\shkb$ is a special case. We specifically address the following two issues. We assume familiarity with the shifted QR algorithm and standard background in numerical analysis; see \cite[Section 1]{banks2021global} and the references therein for more detail.
\begin{enumerate}

    \item {\em Forward Stability of QR Steps.} Consider a degree $k$ shifted QR step:
$$ p(H)=QR\qquad \next{H} = Q^*HQ,$$
where $p(z)=(z-r_1)\ldots(z-r_k)$ is a monic polynomial of degree $k$ and $H$ is an upper Hessenberg matrix.
It is well-known that such a step can be implemented in a way which is backward stable, in the sense that the finite arithmetic computation produces a matrix $\next{H}$ which is the unitary conjugation of a matrix near $H$ \cite{tisseur1996backward}. Backward stability is sufficient to prove correctness of the shifted QR algorithm in finite arithmetic, i.e., whenever it converges in a small number of iterations, the backward error is controlled.
However, it is insufficient for proving an upper bound on the number of iterations before decoupling,\footnote{As in \cite{banks2021global}, we call an (upper) Hessenberg matrix $H$ $\delta$-decoupled if $|H(i+1,i)|\le \delta\|H\|$ for some $i$.} which requires showing that certain subdiagonal entries of the Hessenberg iterates decay rapidly --- to reason about these entries, some form of forward stability is required. The issue is that a shifted QR step is {\em not} forward stable when $p(H)$ is nearly singular (which can occur before decoupling). Thus, the existing convergence proofs break down in finite arithmetic whenever this situation occurs. As far as we know, there is no complete and published proof of rapid convergence of the implicitly shifted QR algorithm with any shifting strategy in finite arithmetic, even on symmetric matrices (see Section \ref{sec:related} for a detailed discussion).

    \item {\em Computation of Approximate Ritz Values.} The Ritz values of order $k$ of an upper Hessenberg matrix $H$ are equal to the eigenvalues of its bottom right $k\times k$ corner $H_{(k)}$; they are also defined variationally as the zeros of the monic degree $k$ polynomial $p_k$ minimizing
    $ \|e_n^* p_k(H)\|$, where $e_n$ is an elementary basis vector.
    All of the higher order shifting strategies we are aware of are defined in terms of these Ritz values. However, we are not aware of any theoretical analysis of how to compute the Ritz values (approximately) in the case of nonsymmetric $H_{(k)}$, nor a theoretical treatment of which notion of approximation is appropriate for their use in the shifted QR iteration.\footnote{In practice, and in the current version of LAPACK, the prescription is to run the shifted QR algorithm itself on $H_{(k)}$, but there are no proven guarantees for this approach.}
\end{enumerate}

These two issues are closely related. A natural strategy for obtaining forward stability is to perturb the zeros $r_1,\ldots, r_k$ of the shift polynomial $p(z)$ so that they avoid the eigenvalues of $H$. Such a perturbation must be large enough to ensure forward stability, but small enough to preserve the convergence properties of the QR iteration, which are presumably tied to the $r_1,\ldots,r_k$ being approximate Ritz values. The precise notion of ``approximate'' thus determines how constrained we are in choosing our shifts while maintaining good convergence properties. 

\subsection{Results and Organization}
This paper contains the following two principal contributions, which together provide a solution to both (1) and (2) for a wide class of shifting strategies.

We use $\kappa_V(M)$ and $\gap(M)$ to refer to the eigenvector condition number and minimum eigenvalue gap of a matrix $M$, respectively. We will assume throughout that we are working with Hessenberg $H$ satisfying $\gap(H)>0$ and consequently $\kappa_V(H)<\infty$; these assumptions can be satisfied with good quantitative bounds at the cost of a small backward error by adding a random perturbation to $H$, as discussed in \cite[Remark 1.4]{banks2021global}. 
\paragraph{(i) Forward Stability by Regularization.} We handle the first issue above simply by replacing any given shifts $r_1,\ldots,r_k$ in a QR step by random perturbations $r_1+w_1,\ldots,r_k+w_k$ where the $w_i$ are independent random numbers of an appropriate size (which depends on $\kappa_V(H)$ and $\gap(H)$). We refer to this technique as {\em shift regularization} and show in Section \ref{sec:regularization} (Lemma \ref{lem:fixguarantee1}) that it yields forward stability of an implicit QR step with high probability, for any Hessenberg matrix $H$ with an upperbound on $\kappa_V(H)$ and a lowerbound on $\gap(H)$, and any shifts $r_1,\ldots,r_k$. 

The proof of forward stability requires us to establish stronger backward stability of implicit QR steps than was previously recorded in the literature; this appears in Sections \ref{sec:backwardiqr} and \ref{sec:forwardiqr} and may be of independent interest.

\paragraph{(ii) Optimal Ritz Values/Early Decoupling Dichotomy.} The second issue is more subtle. The notion of approximate Ritz values relevant for analyzing $\shkb$ is the following variational one. Recall from \cite[Definition 1.2]{banks2021global} that $\{r_1,\ldots,r_k\}\subset \bC$ is called a set of  {\em $\theta$-optimal Ritz values} of a Hessenberg matrix $H$ if:
\begin{equation}\label{eqn:optdef}
    \|e_n^* (H-r_1)\ldots(H-r_k)\|^{1/k}\le \theta \min_p \|e_n^* p(H)\|^{1/k},
\end{equation}
where the minimization is over monic polynomials of degree $k$. Thus, the true Ritz values are $1$-optimal.

It is not immediately clear how to efficiently compute a set of $\theta$-optimal Ritz values, so we reduce this task to the more standard one of computing forward-approximate Ritz values, which are just forward-approximations of the eigenvalues of $H_{(k)}$ with an appropriately chosen accuracy parameter $\beta$ roughly proportional to the right hand side of \eqref{eqn:optdef}. Our key result (Theorem \ref{lem:dichotomy}) is the following {\em dichotomy}: if a set of $\beta$-forward approximate Ritz values $r_1,\ldots,r_k$ of $H$ is {\em not} $\theta$-optimal, then one of the Ritz values $r_j$ must be close to an eigenvalue of $H$ and the corresponding right eigenvector of $H$ must have a large inner product with $e_n$. In the latter scenario we show that a single degree $k$ implicit QR step using the culprit Ritz value $r_j$ as a shift must lead to immediate decoupling, which we refer to as {\em early decoupling.}

Importantly, this dichotomy is compatible with the random regularizing perturbation used in (i), since the property of being a $\beta$-forward approximate Ritz value is preserved (with a slight increase in $\beta$) under small perturbations $r_i\rightarrow r_i+w_i$ when $|w_i|\ll \beta$. Thus, as long as we can compute $\beta$-forward approximations $r_1,\ldots,r_k$ of the eigenvalues of $H_{(k)}$, the combination of (i) and the dichotomy guarantees that with high probability, $r_1+w_1,\ldots,r_k+w_k$ are $\theta$-optimal Ritz values {\em and} the corresponding QR step is forward stable (which is exactly what is needed in order to analyze convergence of the iteration) --- {\em or} we achieve early decoupling.\\

\begin{example}[Necessity of Forward Error for Ritz Value Optimality] It is natural to ask whether the weaker property of being a $\beta$-backward approximation of the eigenvalues of $H_{(k)}$ is sufficient for producing $O(1)$-optimal Ritz values when the right hand side of \eqref{eqn:optdef} is of scale $\beta$. The following example shows that this is not in general the case: let $T$ be an $n\times n$ Hessenberg Toeplitz matrix with $1$s on the superdiagonal, $\delta$s on the subdiagonal, and $T(1,n)=1$.  Let the bottom right $k\times k$ corner of $T$ be $T_{(k)}$ and let $T_{(k)}'=T_{(k)}+\beta e_ke_1^*$. An explicit computation of characteristic polynomials shows that if $r_1,\ldots,r_k$ are the eigenvalues of $T_{(k)}$ and  $r_1'\ldots,r_k'$ are the eigenvalues of $T_{(k)}'$ (which are $\beta$-backward approximations of the $r_i$) then
$$\delta=\|e_n^*(T-r_1)\ldots (T-r_k)\|^{1/k}\ll \|e_n^*(T-r_1')\ldots(T-r_k')\|^{1/k}\approx \beta^{1/k},$$
unless $\beta=O(\delta^k)$. But this latter condition is enough to guarantee that the $r_1'\ldots,r_k'$ are $\delta$-forward approximations of the Ritz values of $T$, which is what we require. Since $T$ is close to normal when $\delta\ll 1/n$, this example also highlights that while we may have control of the nonnormality of $H$, this does not imply any control on the nonnormality of $H_{(k)}$ in general.
\end{example}
\noindent To produce a complete eigenvalue algorithm, we also need the following auxiliary ingredients. 

\paragraph{(iii) Approximating the Eigenvalues of Small Matrices.} In order to carry out (ii), we require an efficient way to obtain forward approximations to the eigenvalues of the small $k\times k$ matrix $H_{(k)}$. In addition, our degree $k$ shifting strategy cannot decouple matrices of size $k\times k$ or smaller, so we also need an algorithm to compute approximations to the eigenvalues of small matrices, to use once we have deflated to a sufficiently small matrix. We will assume access to a black box algorithm $\smalleig$ for use in these two situations, with the following guarantee on a matrix $M$ of dimension $k$ or smaller. (The notion of forward error here is absolute, instead of relative --- this will simplify some of the analysis later on.)

\begin{definition}
    \label{def:small-eig}
    A \textit{small eigenvalue solver} $\smalleig(M,\forward,\phi)$ takes as input a matrix $M$ of size at most $k\times k$, and with probability at least $1-\phi$, outputs $\ax\lambda_1,...,\ax\lambda_k \in \bC$ such that $|\ax\lambda_i - \lambda_i| \le \beta$ for each of $\lambda_1,...,\lambda_k \in \Spec M$.
\end{definition}

We were unable to find a suitably strong and precise worst-case running time bound in the literature for the forward error eigenproblem in the sense of Definition \ref{def:small-eig}.\footnote{One option is to combine the spectral bisection algorithm of \cite{banks2020pseudospectral} (which produces \textit{backward} approximate eigenvalues) with \cite[Theorem 39.1]{bhatia2007perturbation} (which shows that $\Omega(\beta^k)$ backward approximate eigenvalues are $O(\beta)$ forward approximate), but this uses roughly $O(k^4\log^4(k/\beta)\log(k))$ bits of precision, which is larger than we would like. Another possibility is to compute the characteristic polynomial and use polynomial root finders such as \cite{pan2002univariate}, but we could only find rigorous proofs about such algorithms in models of arithmetic other than floating point.} We will use the following result from Part III of this series \cite[Corollary 1.3]{banks2022III}, which may be of independent interest. Note that any other algorithm with provable guarantees may be used in its place.

\begin{theorem}
    \label{thm:smalleig}
    Given $M \in \bC^{k \times k}$ with $\|M\| \le \scale$, there is an algorithm, $\smalleig$, which solves the forward eigenvalue problem in the sense of Definition \ref{def:small-eig},
    using at most 
    $$
        O\big(k^5 \log (k\scale/\beta\phi)^2 + k^2\log(k\scale/\beta\phi)^2 \log (k \log(k\scale/\beta\phi))    \big)
    $$
    arithmetic operations on a floating point machine with $O(k^2 \log(k\scale/\beta \phi)^2)$ bits of precision. 
\end{theorem}

\noindent Note that the algorithm $\smalleig$ uses higher precision than we require anywhere else in this paper, but because it is called infrequently and on $k\times k$ matrices only, the total Boolean operations are still  subdominant (see Remark \ref{rem:provable}).

\paragraph{(iv) Deflation.} Once the shifting strategy $\shkb$ has been used to achieve decoupling, it is typical to \textit{deflate} the resulting matrix by zeroing out small subdiagonal elements. The outcome of this procedure is a block upper triangular matrix whose diagonal blocks are themselves upper Hessenberg, allowing one to recursively apply $\shkb$. Because our analysis of $\shkb$ relies on $\kappa_V(H)$ and $\gap(H)$ being controlled, it is critical that we can preserve these quantities when deflating and passing to a submatrix. This is handled in Section \ref{sec:preserve}.

Finally, we combine (i-iv) above in Section \ref{sec:fullalg} in order to give a fully proven shifted QR algorithm using the strategy $\shkb$.

\begin{theorem}\label{thm:main}
    Let $H$ be an $n\times n$ upper Hessenberg matrix and $\K \ge 2\kappa_V(H)$ and $\Gamma \le \gap(H)/2$ upper and lowerbounds on its eigenvector condition number and minimum eigenvalue gap. For a certain $k = O(\log\K \log\log\K)$ --- chosen as in \eqref{eqn:setk} --- the shifting strategy $\Sh_{k,\K}$ can be implemented in finite arithmetic to give a randomized shifted QR algorithm, $\shqr$, with the following guarantee: for any $\acc > 0$ $\shqr(H,\acc,\phi)$ produces the eigenvalues of a matrix $H'$ with $\|H - H'\| \le \acc\|H\|$, with probability at least $1 - \phi$, using
    \begin{itemize}
        \item $O\left(n^3
        \left(\log\frac{n\K}{\acc\Gamma}\cdot k\log k + k^2\right)\right)$ arithmetic operations on a floating point machine with $O\left(k\log \frac{n\K}{\acc\Gamma \phi}\right)$ bits of precision; and
        \item $O(n\log\frac{n\K}{\acc\Gamma})$ calls to $\smalleig$ with accuracy $\Omega(\frac{\acc^2\Gamma^2}{n^4\K^4\scale})$ and failure probability tolerance $\Omega\left(\frac{\phi}{n^2\log \frac{n\K}{\acc\Gamma}}\right)$
    \end{itemize}
\end{theorem}
\begin{remark}[Constants]
The constants on arithmetic operations and precision hidden in the asymptotic notation above are modest and can be read off by unpacking the expressions for $T_{\shqr}$ in equation \eqref{eq:t-shqr} and $\mach_{\shqr}$ in equation \eqref{eq:mach-requirement-shqr}, respectively.
\end{remark}

\begin{remark}[Computing Eigenvalues of an Arbitrary Matrix]\label{rem:provable}
    The algorithm $\shqr$ can be used to compute backward approximations of the eigenvalues of an arbitrary matrix $A\in \bC^{n\times n}$ with a backward error of $\delta\|A\|$ as follows:
    \begin{enumerate}
        \item Add a random complex Gaussian perturbation of norm $\delta\|A\|/2$ to the input matrix, which yields
    $\log (B/\Gamma) = O(\log (n/\delta))$ with high probability (see \cite[Remark 1.4]{banks2021global}) 
    \item Put the resulting matrix in Hessenberg form using Householder reflectors. This step is backward stable when performed in finite arithmetic \cite{tisseur1996backward}, and thus approximately preserves the bounds on $B,\Gamma$ by the results of Section \ref{sec:preserve}.
    \item Apply Theorem \ref{thm:main} with accuracy $\delta/2$, noting that the bound on $\log(B/\Gamma)$ from step $1$ implies that $k=O(\log(n/\delta)\log\log(n/\delta))$ is sufficient.  
    \end{enumerate}
    This yields a total worst-case complexity bound of $O(n^3\log^2(n/\delta)(\log\log(n/\delta))^2)$
    arithmetic operations with $O(\log^2(n/\delta)\log\log(n/\delta))$ bits of precision {\em plus} $O(n\log(n/\delta)\cdot \log^7(n/\delta)\log\log(n/\delta)^5)$ operations with $O(\log^4(n/\delta)(\log\log(n/\delta)^2))$ bits of precision for the calls to $\smalleig$. The Boolean cost of calls to $\smalleig$ is subdominant whenever $n\ge \log^{7/2}(n/\delta)(\log\log (n/\delta))^{2}$.
    \end{remark}
    
    While this asymptotic complexity guaranteed by Remark \ref{rem:provable} is significantly higher than the nearly matrix multiplication time spectral bisection algorithm of \cite{banks2020pseudospectral}, that algorithm uses  $O(\log^{4}(n/\delta)\log(n))$ bits of precision throughout the algorithm, moreover with a larger hidden constant. On the other hand, the algorithm of \cite{abbcs} uses $O(n^{10}/\delta^2)$ arithmetic operations but with only $O(\log(n/\delta))$ bits of precision (as is stated but not formally proven in \cite{abbcs}).

\begin{remark}[Hermitian Matrices]\label{rem:hermitian}
For the important case of Hermitian tridiagonal matrices there is no difficulty in maintaining $\kappa_V(H)=1$, so we may take $k=2$ and $B=1$. A minimum eigenvalue gap of $\Gamma\ge (\delta/n)^c$ may be guaranteed by adding a diagonal Gaussian perturbation of size $\delta/2$ \cite{aizenman2017matrix} to the matrix (or by adding a GUE perturbation and then tridiagonalizing the matrix). The Ritz values in this case can be computed exactly using the quadratic formula. The amount of precision required by Theorem \ref{thm:main} is consequently simply $O(\log(n/\delta))$ and the number of arithmetic operations used is $O(n^3+n^2\log(n/\delta))$, which is asymptotically the same as in the exact arithmetic analysis of tridiagonal QR with Wilkinson shift. 
\end{remark}

\subsection{Related Work}\label{sec:related}
The need for a finite arithmetic convergence analysis of shifted QR in the case of symmetric tridiagonal matrices was noted in the remarkable thesis of Sanderson \cite{sanderson1976proof}, who observed that it does not follow from the exact arithmetic analysis of Wilkinson \cite{wilkinson1968global}. Sanderson formally proved the convergence of the tridiagonal QR algorithm with explicit (as opposed to implicit) QR steps using Wilkinson shift under certain additional assumptions, one of which \cite[Section 4]{sanderson1976proof} is that the ``computation of the [Wilkinson shift] be done more accurately [i.e., in exact arithmetic]''. Sanderson left open the question of analyzing implicit shifted QR and gave an example for which its convergence  breaks down unless the machine precision is sufficiently small in relation to the subdiagonal entries of the matrix. These insightful observations of Sanderson are consistent with the approach taken in this paper, and Sanderson's question is resolved by Remark \ref{rem:hermitian}, albeit with a different shifting strategy.\\

\noindent {\em Forward Stability of Shifted QR.} An important step towards understanding and addressing the two issues mentioned at the beginning of the introduction was taken by Parlett and Le \cite{parlett1993forward}, who showed that for symmetric tridiagonal matrices, high sensitivity of the next QR iterate to the shift parameter (a form of forward instability) is always accompanied by ``premature deflation'', which is a phenomenon specific to ``bulge-chasing'' implementations of the implicit QR algorithm on tridiagonal matrices. Our dichotomy is distinct from but was inspired by their paper, and carries the same conceptual message: if the behavior of the algorithm is highly sensitive to the choice of shifts, then one must already be close to convergence in some sense. 

Watkins \cite{watkins1995forward} argued informally (but did not prove) that the implicit QR iteration should in many cases converge rapidly even in the presence of forward instability. This is an intriguing direction for further theoretical investigation, and could potentially lead to provable guarantees for the shifted QR algorithm with lower precision than required in this paper (see the discussion in Section \ref{sec:conclusion}). \\

\noindent {\em Aggressive Early Deflation.} The classical criterion for decoupling/deflation in shifted QR algorithms is the existence of small subdiagonal entries of $H$. The celebrated papers \cite{braman2002multishift,braman2002multishift2} introduced an additional criterion called aggressive early deflation which yields significant improvements in practice. Kressner \cite{kressner2008effect} showed that this criterion is equivalent to checking for converged Ritz values (i.e., Ritz pairs which are approximate eigenpairs of $H$), and ``locking and deflating them'' (i.e., deflating while preserving the Hessenberg structure of $H$) using Stewart's Krylov-Schur algorithm \cite{stewart2002krylov}.

The early decoupling procedure introduced in this paper is similar in spirit to aggressive early deflation --- in that it detects Ritz values which are close to eigenvalues of $H$ and enables decoupling even when the subdiagonal entries of $H$ are large --- but different in that it does not require the corresponding Ritz vector to have a small residual, and it ultimately produces classical decoupling in the sense of a small subdiagonal entry.\\

\noindent {\em Shift Blurring.} The shifting strategies considered in \cite{banks2021global} and in this paper use shift polynomials $p(z)=(z-r_1)\ldots(z-r_k)$ of degree $k$ where $k$ is roughly proportional to $\log\kappa_V(H)$.  It was initially proposed \cite{bai1989block} that such higher degree shifts should be implemented via ``large bulge chasing'', a procedure which computes the $QR$ decomposition of $p(H)$ in a single implicit QR step. This procedure was found to have poor numerical stability properties, which was referred to as ``shift blurring'' and explained by Watkins \cite{watkins1996transmission} and further by Kressner \cite{kressner2005use} by relating it to some ill-conditioned eigenvalue and pole placement problems. 

To avoid these issues, we implement all degree $k$ QR steps in this paper as a sequence of $k$ degree-$1$ ``small bulge'' QR steps. However, since our analysis requires establishing forward stability of each degree $k$ step, the amount of numerical precision required for provable $\delta-$decoupling increases as a function of $k$, roughly as $O(k\log(n/\delta))$ bits. This increase in precision is sufficient to avoid shift blurring. We suspect that forward stability of large bulge chasing can be established given a similar increase in precision, and leave this as a direction for further work.

\subsection{Discussion}
\label{sec:conclusion}
The purpose of this paper is to provide a framework for rigorous convergence analysis of shifted QR algorithms in finite arithmetic, and to show using this framework that the shifting strategy of \cite{banks2021global} enjoys reasonable complexity and precision bounds. The main resource we have tried to optimize is the amount of precision used, since this seems to be the main bottleneck in turning ``theoretical'' algorithms into practical ones. The parameters have been optimized to enable provably good worst case performance rather than ``real world'' performance; indeed, while our worst case bounds are quite good from a theoretical perspective, they are still far from the performance desired from software libraries. The specific implementation in this paper is accordingly {\em not} a prescription for an actual software implementation. Rather, it is best viewed as a framework for further experimentation and engineering, with the goal of eventually obtaining a practically competitive algorithm with a rigorous proof of correctness and worst case complexity. 

In the meantime, one way that the algorithm in this paper may be used profitably in practice is as a final exceptional shift in existing shifted QR implementations (which already have a rather long list of shifts to try in cases of slow convergence). This final exceptional shift will be invoked very rarely, thereby hardly affecting typical performance, but will nonetheless guarantee convergence in all cases at the cost of occasionally using higher precision and running time.

The main question which remains from a theoretical perspective is to reduce the bits of precision required from $O(\log^2(n/\delta)\log\log(n/\delta))$ to $O(\log(n/\delta))$ in the context of Remark \ref{rem:provable}, ideally with a small constant term, which would be asymptotically optimal and in line with the standard notion of ``numerical stability'' in numerical analysis. The main bottleneck to doing this is that our current proof requires establishing forward stability of a sequence of $k=O(\log(n/\delta)\log\log(n/\delta))$ degree $1$ QR steps, which necessitates $O(k\log(n/\delta))=O(\log^2(n/\delta)\log\log(n/\delta))$ bits of precision. Whether this can be avoided is an an interesting conceptual question. A second, related, question is to reduce the number of bits of precision required for computing optimal Ritz values.
\section{Preliminaries}\label{sec:finite}
\label{sec:qrpreliminaries}
All vector norms are $\ell_2$, and all matrix norms are the induced $\ell_2$ operator norm, unless otherwise specified. We denote the distance between two sets $\calR,\calS \subset\bC$ as
$$
    \dist(\calR,\calS) := \inf_{r \in \calR,\, s \in \calS}|r - s|.
$$

\paragraph{Finite Precision Arithmetic.}
We use the standard floating point axioms from \cite[Chapter 2]{higham2002accuracy} (ignoring overflow and underflow as is customary), and use $\mach$ to denote the unit roundoff. Specifically, we will assume that we can add, subtract, multiply, and divide floating point numbers, and take square roots of positive floating point numbers, with relative error $\mach$.

Our implementation of implicit QR steps is based on \textit{Givens rotations}. If $x \in \R^2$, write $\giv(x)$ for the $2\times 2$ Givens rotation mapping $\giv(x) : x \mapsto \|x\|e_1$. It is routine \cite[Lemmas 19.7-19.8, e.g.]{higham2002accuracy} that, assuming $\mach \le 1/24$, one can compute the norm of $x$ with relative error $2\mach$ and apply $\giv(x)$ to a vector $y \in \R^2$ in floating point so that
$$
    \left|(\ax{\giv(x) y})_i - (\giv(x) y)_i\right| \le \|y\|\frac{6\mach}{1 - 6\mach} \le \|y\| \cdot 8\mach \qquad i=1,2.
$$

For some tasks, our algorithm and many of its subroutines need to set certain scalar parameters  in order to know when to halt, at what scale to perform certain operations and how many iterations to perform. In this context, sometimes the algorithm will have to compute $k$-th roots for moderate values of $k$ --- even though these operations are not directly used on the matrices in question. We will assume that the following elementary functions can be computed accurately and relatively quickly. 

\begin{lemma}[$k$th Roots]
\label{lem:elementaryfunctions}
There exist small universal constants $\Croot, \croot \geq 1$, such that whenever $k \croot  \mach \leq \epsilon \leq 1/2 $ and for any $a\in \bR^+$, there exists an algorithm that computes $a^{1/k}$ with relative error $\epsilon$ in at most $$T_{\rt}(k,\epsilon):= \Croot k \log(k\log(1/\epsilon))$$ arithmetic operations.
\end{lemma}

\begin{proof}[Sketch]
Use Newton's method, with starting point found via bisection.
\end{proof}

\paragraph{Random Sampling Assumptions.} As discussed above, we will repeatedly regularize our shifts by replacing each with uniformly random point on a small surrounding disk of radius $O(\acc^2)$, where $\acc$ is the accuracy. To simplify the presentation, we will assume that these perturbations can be executed in \textit{exact} arithmetic. Importantly, this assumption's only impact is on the failure probability of the algorithm, and its effect is quite mild. We will see below that the algorithm fails when one of our randomly perturbed shifts happens to land too close to an eigenvalue, and we bound the failure probability by computing the area of the `bad' subset of the disk where this occurs. If the random perturbation was instead executed in finite arithmetic, the probability of landing in the bad set differs from this estimate by $O(\mach/\acc^2)$. Since we will set $\mach = o(\acc^2)$, this discrepancy can reasonably be neglected.

\begin{definition}[Efficient Perturbation Algorithm]
    An efficient random perturbation algorithm takes as input $r \in \bC$ an $R>0$, and generates a random $w\in \bC$ distributed uniformly in the disk $\mathbb{D}(r, R)$ using $\cd$ arithmetic operations. 
\end{definition}

\subsection{Key Definitions and Lemmas from \cite{banks2021global}}

For a Hessenberg matrix $(h_{ij})_{i, j=1}^n =H \in \bC^{n\times n}$ we define the \emph{potential} of $H$ as
$$\pot(H):=|h_{n-k, n-k-1}\cdots h_{n, n-1}|^{1/k}, $$
and we will use this quantity to track the convergence of the QR iteration. We will also use $\chi_k(z)$ to denote the characteristic polynomial of $\corner{H}{k}$, the lower-right $k\times k$ corner of $H$, and as mentioned in the introduction, we will exploit that the Ritz values of $H$ are the roots of $\chi_k$ and that the following variational characterization exists (see \cite[Lemma 2.2]{banks2021global} for a proof). 

\begin{lemma}[Variational Formula for $\pot$]
    \label{lem:minnorm}
    Let $H\in \bC^{n\times n}$ be any Hessenberg matrix. Then, for any $k$ 
    $$
        \psi_k(H) = \min_{p\in \calP_k}\|e_n^* p(H)\|^{1/k},
    $$
    with the minimum attained for $p= \chi_{k}$. 
\end{lemma}

Our analysis will heavily rely on the notion of {\em approximate functional calculus} introduced in \cite{banks2021global}, which for convenience of the reader we recall here. First, consider a diagonalizable Hessenberg matrix $H\in \bC^{n\times n}$, with  diagonalization $H=VDV^{-1}$ for $V$ chosen\footnote{If there are multiple such $V$, choose one arbitrarily.} to satisfy $\|V\|= \|V^{-1}\| = \sqrt{\kappa_V(H)}$ and let $\lambda_i = D_{ii}$ be the $i$-th eigenvalue of $H$. We will use $Z_H$ to denote the random variable supported on $\Spec(H)$ with distribution
\begin{equation}\label{eqn:specmeasure}\P[Z_H = \lambda_i ] = \frac{|e_n^* V e_i|^2}{\|e_n^* V\|^2}.\end{equation}
We will often use the following inequalities (see \cite[Lemma 2.4]{banks2021global} for a proof). 

\begin{lemma}[Approximate Functional Calculus]
    \label{lem:spectral-measure-apx}
    For any upper Hessenberg $H$ and complex function $f$ whose domain includes the eigenvalues of $H$,
    $$
        \frac{\|e_n^\ast f(H)\|}{\kappa_V(H)} \le \E\left[|f(Z_H)|^2\right]^{\frac{1}{2}} \le \kappa_V(H)\|e_n^\ast f(H)\|.
    $$
\end{lemma}

As in \cite{banks2021global}, if $\calR = \{r_1,...r_k\} \subset \bC$ and $p(z) = (z-r_1)\cdots(z-r_k)$, we will call $r \in \calR$ is \textit{$\cp$-promising} if
$$
    \E\frac{1}{|Z_H - r|^k} \ge \frac{1}{\cp^k}\E \frac{1}{|p(Z_H)|}.
$$
This notion will always be applied to a set $\calR$ of $\r$-optimal Ritz values in the sense of \eqref{eqn:optdef}. A key observation from \cite{banks2021global}, which we will recycle here, is that if a shifted QR step using an $\cp$-promising, $\r$-optimal Ritz value $r$ does not make progress, then $Z_H$ has support on a disk of radius $\approx \cp \pot(H)$ about $r$ --- which in particular means that there is a nearby eigenvalue whose left eigenvector is aligned with $e_n^\ast$.

\begin{lemma}[Stagnation Implies Support]
    \label{lem:main}
    Let $\gamma,\theta\in (0, 1)$ and let $\calR$ be a set of $k$ $\r$-approximate Ritz values of $H$. Suppose $r \in \calR$ is $\cp$-promising and assume 
    \begin{equation}
        \label{eqn:jstag} 
        \psi_k\left(\exactqr(H,(z - r)^k)\right)\geq (1-\gamma)\psi_k(H)>0.
    \end{equation}
    Then, for every $t\in (0,1)$,
    \begin{equation}
        \nonumber
        \P \left[ |Z_H-r|\le (1+\r)\cp\left(\frac{\kappa_V(H)}{t}\right)^{1/k}\pot(H)\right]
        \ge (1-t)^2\frac{(1-\gamma)^{2k}}{\cp^{2k}(1+\r)^{2k}\kappa_V(H)^{4}}.
    \end{equation}
\end{lemma}

\noindent In fact, one can verify from the proof of our Lemma \ref{lem:main} in \cite[Lemma 2.8]{banks2021global} that the hypothesis \eqref{eqn:jstag} may be replaced with the weaker condition
\begin{equation}
    \label{eqn:jstag-alt}
    \|e_n^\ast (H - r)^{-k}\|^{1/k} \ge (1- \gamma)\pot(H),
\end{equation}
which we will find more useful here.

\subsection{Reader Guide and Parameter Settings} There are many algorithm inputs, constants, and parameters that the reader will encounter; we will collect them here, along with some typical settings. We will regard our main algorithm $\shqr$ in fact as a family of algorithms, indexed by several defining parameters; these in turn used to set a number of global constants used by the algorithm and its subroutines. The most important of the former is the ``non-normality'' or condition number bound $\K$, from which we define the shift degree $k$ to be the smallest power of $2$ for which
\begin{equation}\label{eqn:setk} 
    \K^{\frac{8\log k+3}{k-1}}\cdot (2\K^4)^{\frac{2}{k-1}}\le 3,
\end{equation}
which makes $k = O(\log\K\log\log\K)$. We further define the auxiliary constants
\begin{equation}\label{eqn:settheta}
    \cp:=(1.01\K)^{4\log k/k}\in [1,2],\qquad\r:=\tfrac{1.01}{0.998^{1/k}}(2\K^4)^{1/2k} \in [1,2]\qquad \gamma:=0.2,
\end{equation}
which depend only on $B$.
\begin{table}[h]
\centering
\begin{tabular}{ l l l l}
    \toprule 
    Defining Parameter & Meaning & Typical Setting \\
    \midrule 
    $\K$ & Eigenvector Condition Number Bound & $\K \ge 2\kappa_V(H)$ \\
    $\Gamma$ & Minimum Gap Bound & $\Gamma \le \gap(H)/2$ \\
    $\scale$ & Operator Norm Bound & $\scale \ge 2\|H\|$ \\
    $k$ & Shift Degree & $O(\log\K\log\log\K)$
     \\
    \midrule
    Global Constant & & \\
    \midrule 
    $\cp$ & Ritz Value Promising-ness & $\cp\in [1,2]$ \\
    $\r$ & Ritz Value Optimality & $\r\in[1,2]$ \\
    $\gamma$ & Decoupling Rate & $0.2$ \\
    \bottomrule
\end{tabular}
\caption{Global Data for $\shqr$}
\label{table:qrii-global-data}
\end{table}

Table \ref{table:shqr-params} contains the input parameters for $\shqr$, as well as internal parameters used by its subroutines. The setting of the working accuracy below is to ensure that the norm, eigenvector condition number, and minimum eigenvalue gap are controlled for every matrix $H'$ encountered in the course of the algorithm, in the sense that
$$
    \kappa_V(H') \le 2\kappa_V(H) \le \K \qquad \|H'\| \le 2\|H\| \le \scale \qquad \gap(H') \ge \gap(H)/2 \ge \Gamma.
$$
We will not include the defining parameters or global constants as input to $\shqr$ or its subroutines, and instead assume that all subroutines have access to them; however, we will for clarity keep track of which of this \textit{global data} each subroutine uses, and any constraints that it places on their inputs. Table \ref{table:shqr-subs}  lists the main subroutines (note that we will write $\Sh_{k,\K}$ for the finite arithmetic implementation of $\shkb$).

\begin{table}[h]
\centering
\begin{tabular}{lll}
    \toprule 
    Input Parameter & Meaning & Typical Setting \\
    \midrule
    $H$ & Upper Hessenberg Matrix & {} \\
    $\acc$ & Accuracy &  \\
    $\phi$ & Failure Probability Tolerance &  \\
    && \\
    
    Internal Parameter && \\
    \midrule 
    $\wacc$ & Working Accuracy & $\Omega\left(\min\{\acc n^{-2},\gapbound n^{-3/2}\K^{-2}\}\right)$ \\
    $\varphi$ & Working Failure Probability Tolerance & $\Omega\left(\tfrac{\phi}{\log(\wacc/\scale)}\right)$ \\
    $\tol,\pretol$ & Regularization Parameters & $\Omega(\wacc^2), \Omega(\wacc^2\phi^{-1/2}\scale^{-1})$ \\
    $\forward$ & Forward Accuracy for Ritz Values & $\Omega(\wacc^2\scale^{-1})$ \\
    $\calR$ & Approximate Ritz Values & \\
    $\calS$ & Exceptional Shifts & \\
    \bottomrule
\end{tabular}
\caption{Input and Internal Parameters for $\shqr$}
\label{table:shqr-params}
\end{table}

\begin{table}[h]
\centering
\begin{tabular}{lllll}
    \toprule
    Subroutine & Action & Output & Input & Global Data \\
    \midrule 
    $\iqr$ & Implicit QR Step & $\ax{\next{H}}, \ax R$ & $H,p(z)$ & \\
    $\comptau{m}$ & Approximate $\tau_{p(z)}^m(H) = \|e_n^\ast p(H)^{-1}\|$ & $\ax{\tau^m}$ & $H, p(z)$ & \\
    $\optimal$ & Check Ritz Value Optimality & $\optflag$ & $H,\calR$ & $\r$ \\
    $\regularize$ & Compute $\r$-Optimal Ritz Values & $\next H, \calR, \decouple$ & $H,\wacc,\phi$ & $\scale,\Gamma,\r$ \\
    $\find$ & Find a $\cp$-Promising Ritz Value & $r$ & $H,\calR$ & $\cp$ \\
    $\exc$ & Compute Exceptional Shifts & $\calS$ & $H,r,\wacc,\phi$ & $\K,\scale,\gamma,\r,\cp$ \\
    $\shkb$ & Shifting Strategy to Reduce $\pot(H)$ & $\next H$ & $H, \calR, \wacc, \phi$ & $\K,\scale,\gamma,\r,\cp$ \\
    $\deflate$ & Deflate a Decoupled Matrix & $H_1,H_2,...$ & $H,\wacc$ & \\
    \bottomrule
\end{tabular}
\caption{Subroutines of $\shqr$}
\label{table:shqr-subs}
\end{table}

\paragraph{Absolute vs. Relative Decoupling.} Because $\shqr$ and its subroutines do not have direct access to the norms of matrices, we will find it useful for the remainder of the paper to work with an \textit{absolute} notion of decoupling, instead of the relative one used in \cite{banks2021global}. In particular, we will say that a matrix $H$ is \textit{$\wacc$-decoupled} if one of its $k$ bottom subdiagonal entries is smaller than $\wacc$ (as opposed to $\wacc\|H\|$), and \textit{$\wacc$-unreduced} if every one of its $k$ bottom subdiagonal entries is larger than $\wacc$.
\section{Implicit QR: Implementation, Forward Stability, and Regularization}

In Section \ref{sec:backwardiqr} we present a standard implementation (called ``$\iqr$'') of a degree $1$ (i.e., single shift) implicit QR step  using Givens rotations (see \cite[Section 4.4.8]{demmel1997applied}) and provide an analysis of its backward stability which is slightly stronger than the guarantees of \cite{tisseur1996backward}\footnote{\cite{tisseur1996backward} uses Householder reflectors instead of Givens rotations. We have chosen the latter for simplicity of exposition, but the stronger backward stability analysis obtained in Lemma \ref{lem:iqr-single-guarantees} can also be shown for Householder reflectors. }. We then use this to give a corresponding backward error bound for a degree $k$ $\iqr$ step. We suspect much of this material is already known to experts, but we could not find it in the literature so we record it here.

In Section \ref{sec:forwardiqr} we prove bounds on the forward error of a degree $k$ $\iqr$ step in terms of the distance of the shifts to the spectrum; we will accordingly refer to shifts which are appropriately far away from the spectrum as {\em forward stable}. We also record a forward error bound on the bottom right entry $R_{nn}$ of the QR factorization, which is used in analyzing many shifting strategies.

We show in Section \ref{sec:regularization} that a sufficiently large random perturbation of any choice of shifts is commensurately forward stable, with high probability. 

\label{sec:implicitQR}

\subsection{Description and Backward Stability of $\iqr$}\label{sec:backwardiqr}

We begin with some preliminaries on implicit QR steps in exact arithmetic.
\begin{definition}
    The \textit{QR decomposition} of an invertible matrix $M$ is the unique  factorization $M = QR$ where $Q$ is unitary and $R$ is upper triangular with positive diagonal entries. We will use $$[Q, R] =\qr(M)$$ 
    to signal that $Q$ and $R$ are the matrices coming from the QR decomposition of $M$. 
\end{definition}

Given a polynomial $p(z)$ and a Hessenberg matrix $H$, $\exactqr(H,p(z) )$ will denote the matrix $\next{H} = Q^* H Q$ where $[Q, R] =p(H) $. When $p(z)=z-s$ we will use $\exactqr(H, s)$ as a shorthand notation for $\exactqr(H, z-s)$. We will also denote by $\varkappa(M):= \|M\|\|M^{-1}\|$ the condition number of a matrix $M$. We pause to verify a fundamental composition property of $\exactqr$; the proof is standard (e.g. see \cite[Section 2.3]{tisseur1996backward}), but we will need to adapt it in the sequel so we include it for the reader's convenience.

\begin{lemma}
    \label{lem:exact-iqr-composition}
    For any invertible $H$ and polynomial $p(z) = (z - r_1)\cdots(z-r_k)$,
    \begin{equation}
        \label{eq:exact-iqr-composition}
        \exactqr(H,p(z)) = \exactqr( \cdots \exactqr(\exactqr(H,r_1),r_2),...,r_k).
    \end{equation}
    Moreover, if $[Q, R] = \qr(p(H))$, $H_1 = H$, and for each $\ell \in [k]$ we set $[Q_\ell ,R_\ell]:= \qr(H_\ell - r_\ell) $ and $H_{\ell+1} := Q_\ell^\ast H_\ell Q_\ell$, then
    \begin{equation}
        \label{eq:multishiftedQR}
        Q=Q_1 \cdots Q_k \quad \text{and} \quad R= R_k R_{k-1}\cdots R_1.
    \end{equation}
\end{lemma}

\begin{proof}
    Repeatedly using definition of $Q_\ell$, $R_\ell$, and $H_\ell$ for each $\ell \in [k]$, we can compute
    \begin{align*}
        p(H) = p(H_1) &= (H_1 - r_k)\cdots(H_1 - r_1) \\
        &= (H_1 - r_k) \cdots (H_1 - r_2)Q_1 R_1 & & H_1 - r_1 = Q_1R_1 \\
        &= (H_1 - r_k) \cdots Q_1(H_2 - r_2)R_1 & & H_2 = Q_1^\ast H_1 Q_1 \\
        &= (H_1 - r_k) \cdots (H_1 - r_3)Q_1Q_2R_2R_1 & & H_2 - r_2 = Q_2R_2 \\
        &= Q_1 Q_2 \cdots Q_k R_k R_{k-1} \cdots R_1, & & \text{etc.}
    \end{align*}
    where in the final equality we continue passing $Q_1\cdots Q_\ell$ across the term $H_1 - r_\ell$ and then replace the resulting $H_\ell - r_\ell = Q_\ell R_\ell$. Since each $Q_\ell$ is unitary and $R_\ell$ has positive diagonal entries, uniqueness of the QR decomposition gives $Q = Q_1 \cdots Q_k$ and $R = R_k \cdots R_1$ as desired. The composition property \eqref{eq:exact-iqr-composition} is then immediate.
\end{proof}

The following corollary will be repeatedly useful.

\begin{lemma}
    Under the hypotheses of Lemma \ref{lem:exact-iqr-composition},
    \begin{equation}
        \label{eq:Rs}
        \|e_n^* p(H)^{-1}\|^{-1} = R_{nn} = (R_1)_{nn} \cdots (R_k)_{nn} 
    \end{equation}
\end{lemma}

\begin{proof}
    Maintaining the notation of Lemma \ref{lem:exact-iqr-composition}, we have
    $$
        \|e_n^* p(H)^{-1}\| = \|e_n^* R^{-1} Q^* \| = \|e_n^* R^{-1}\| = \frac{1}{R_{n,n}},
    $$
    and the proof is concluded by observing that \eqref{eq:multishiftedQR} implies $R_{n,n} = (R_1)_{n,n} \cdots (R_k)_{n,n} $.
\end{proof}

We will require the following definition of backward stability for a degree $1$ implicit QR step. The difference between this and the backward stability condition considered in \cite{tisseur1996backward} is the additional second equation below.

\begin{definition}[Backward-Stable Degree $1$ Implicit QR Algorithm]
\label{def:stableiqr}
    A $\muqr(n)$-stable single-shift implicit QR algorithm takes as inputs a Hessenberg matrix $H \in \bC^{n\times n}$ and a shift $s \in \bC$ and outputs a Hessenberg matrix $\ax{\next{H}}$ and an exactly triangular matrix $\ax{R}$, for which there exists a unitary $\ax{Q}$ satisfying
    \begin{align}
        \left\|\ax{\next{H}}-\ax{Q}^\ast H \ax{Q}\right\| &\leq \|H - s\| \muqr(n) \mach \label{eq:iqr-apx-next} \\
        \left\|H - s - \ax{Q}\ax{R}\right\| &\leq \|H - s\|\muqr(n)\mach \label{eq:iqr-apx-qr} 
    \end{align}
\end{definition}

\newcommand{\HessenbergQR}{\mathsf{HessenbergQR}}

We now verify that there is a suitable backward-stable implicit QR algorithm. The pseucodode of $\iqr$ given below is a standard implementation based on Givens rotations. We use {\sf sans serif} fonts to indicate subroutines implemented in finite arithmetic.

\begin{figure}[h]
    \begin{boxedminipage}{\textwidth}
        $$ \iqr $$
        \textbf{Input:} Hessenberg $H$, shift $s \in \bC$ \\
        \textbf{Output:} Hessenberg $\ax{\next H}$ and triangular $\ax{R}$ \\
        \textbf{Ensures:} $\|\ax{\next H}\| \le \|H\| + 32n^{3/2}\mach\cdot\|H -s\|,$ and there exists unitary $\ax Q$ for which $\|\ax{\next{H}} - \ax{Q}^\ast H \ax{Q}\| \le 32 n^{3/2}\mach  \cdot\|H - s\|$ and $\|H - s - \ax{Q} \ax{R}\| \le  16 n^{3/2}\mach \cdot \|H -s \|$
        \begin{enumerate}
            \item $\ax{R} \gets H - s$
            \item \textbf{For} $i = 1,2,...,n-1$
            \begin{enumerate}
                \item $X_{1:2,i} \gets \ax{R}_{i:i+1,i}$
                \item $\ax{R}_{i:i+1,i+1,n} \gets \giv(X_{1:2,i})^\ast\ax{R}_{i:i+1,i+1:n} + E_{2,i,b}$
                \item $\ax{R}_{i:i+1,i} \gets \begin{pmatrix} \|X_{1:2,i}\| + E_{2,i,c} \\ 0 \end{pmatrix}$
            \end{enumerate}
            \item $\ax{\next{H}} \gets \ax{R}$
            \item \textbf{For} $i = 1,2,...n-1$
            \begin{enumerate}
                \item $\ax{\next{H}}_{1:n,i:i+1} \gets \ax{\next{H}}_{1:n,i:i+1}\giv(X_{1:2,i}) + E_{4,i}$
            \end{enumerate}
            \item $\ax{\next{H}} \gets \ax{\next{H}} + s$
        \end{enumerate}
    \end{boxedminipage}
\end{figure}
\begin{lemma}[Backward Stability of Degree $1$ $\iqr$] \label{lem:iqr-single-guarantees}
    Assuming 
    \begin{align}
        \mach \le \min\left\{\frac{1}{24} , \frac{\log 2}{8n^{5/2}}\right\} = 2^{-O(\log n)},
    \end{align}
    $\iqr$ satisfies its guarantees and uses at most $7n^2$ arithmetic operations. In particular, it is a $\muqr(n)$-stable implicit QR algorithm for $\muqr(n) = 32 n^{3/2}$.
\end{lemma}
\begin{proof}[The straightforward proof is deferred to Appendix \ref{sec:implicitQRdeferred}]
\end{proof}.

We now extend the definition of $\iqr$ to   shifts of higher degree. We take the straightforward approach of composing many degree $1$ QR steps to obtain a higher degree one. Given a Hessenberg matrix $H$, an implicit QR algorithm $\iqr$ satisfying Definition \ref{def:stableiqr}, and shifts $s_1,\ldots,s_k$, we will define 
\begin{equation}
    \iqr(H,\{s_1,\ldots,s_k\}) := \iqr(\iqr(\cdots \iqr(\iqr(H,s_1),s_2), \cdots), s_k),
\end{equation}
which can be executed in $T_{\iqr}(n,k) = 7kn^2$ arithmetic operations. We will sometimes use the notation
$$\iqr(H,p(z))=\iqr(H,\{s_1,\ldots,s_k\})$$
where $p(z)=(z-s_1)\ldots(z-s_k)$, though it is understood that $\iqr$ takes the roots of $p$ and not its coefficients as input. Lemma \ref{lem:iqr-single-guarantees} is readily adapted to give backward stability guarantees for $\iqr(H,p(z))$.

\begin{lemma}[Backward Error Guarantees for Higher Degree $\iqr$]
    \label{lem:iqr-multi-backward-guarantees}
    Fix $C > 0$ and let $p(z) = \prod_{\ell \in [k]}(z - s_\ell)$, where $\calS = \{s_1,...,s_k\} \subset \mathbb{D}(0,C\|H\|)$. Write $\Big[\ax{\next{H}},\ax{R}_1,...,\ax{R}_k\Big] = \iqr(H,p(z))$, and let $\ax{Q}_\ell$ be the unitary guaranteed by Definition \ref{def:stableiqr} to the $\ell$th internal call to $\iqr$. Assuming
    $$
        \muqr(n)\mach \le 1/4,
    $$
    the outputs $\ax{R} = \ax{R}_k \cdots \ax{R}_1$ and $\ax{Q} = \ax{Q}_1 \cdots \ax{Q}_k$ satisfy
    \begin{align}
        \left\|\ax{\next{H}} - \ax{Q}^\ast H \ax{Q} \right\| \le 1.4 k(1 + C)\|H\|\muqr(n)\mach \\
        \left\|p(H) - \ax{Q}\ax{R}\right\| \le 4\Big(2(1 + C)\|H\|\Big)^k\muqr(n)\mach.
    \end{align}
\end{lemma}
\begin{proof}[The straightforward proof is deferred to Appendix \ref{sec:implicitQRdeferred}]
\end{proof}

\subsection{Forward Stability of Higher Degree $\iqr$}\label{sec:forwardiqr} 
In this subsection we prove {forward error} guarantees for $\iqr(H,p(z))$ using the backward error guarantees of the previous section. Let us first recall the following bound on the condition number of the QR decomposition \cite[Theorem 1.6]{sun1991perturbation}. 

\begin{lemma}[Condition Number of the QR Decomposition] \label{lem:condQR}
    Let $M, E\in \mathbb{C}^{n\times n}$ with $M$ invertible. Furthermore assume that $\|E\|\|M^{-1}\|\leq \frac{1}{2}$. If $[Q, R]=\qr(M)$  and $[\ax{Q}, \ax{R}] = \qr( M+E)$,  then 
    \[ 
        \| \ax{Q} - Q \|_F \le 4 \Vert M^{-1} \Vert \Vert E \Vert_F \quad \text{and} \quad \|\ax{R}-R\| \leq 3\|M^{-1}\|\|R\|\|E\|.
    \]
\end{lemma}




The main result of this subsection, which will be used throughout the paper, is the following.
\begin{lemma}[Forward Error Guarantees for $\iqr$]
    \label{lem:multiiqrstability}
    Under the hypotheses of Lemma \ref{lem:iqr-multi-backward-guarantees}, and assuming further that $[Q, R]= \qr(p(H))$, $\next{H} = Q^\ast H Q$, and
    \begin{align}
    \label{assum:machvsp}
        \mach \le \mach_{\iqr}(n,k,\|H\|,\kappa_V(H),\dist(\calS,\Spec H)) 
        &:= \frac{1}{8\kappa_V(H)\muqr(n)}\left(\frac{\dist(\calS,\Spec H)}{\|H\|}\right)^k  \\
        &= 2^{-O\left(\log n\kappa_V(H) + k\log\frac{\|H\|}{\dist(\calS,\Spec H)}\right)}, \nonumber
    \end{align}
    we have the forward error guarantees:
    \begin{align}
        \|\ax{Q} - Q\|_F &\le 16 \kappa_V(H)\left(\frac{(2 + 2C)\|H\|}{\dist(\calS,\Spec H)}\right)^k n^{1/2}\muqr(n)\mach \\
        \|\ax{R} - R\| &\le 12 \kappa_V(H)\left(\frac{(2 + 2C)^2\|H\|^2}{\dist(\calS,\Spec H)}\right)^k \muqr(n)\mach \\
        \left\|\ax{\next{H}} - \next{H}\right\|_F &\le 32\kappa_V(H) \|H\|\left(\frac{(2 + 2C)\|H\|}{\dist(\calS,\Spec H)}\right)^k n^{1/2}\muqr(n)\mach.
    \end{align}
\end{lemma}

\begin{proof}
    The first two assertions are immediate from applying Lemma \ref{lem:condQR} to $M = p(H)$, computing
    $$
        \|M^{-1}\| = \|p(H)^{-1}\| \le \frac{\kappa_V(H)}{\dist(\calS,\Spec H)^k},
    $$
    bounding $\|p(H)\| \le (2 + 2C)^k\|H\|^k$, and finally using Lemma \ref{lem:iqr-multi-backward-guarantees} to control $\|E\| \le 2(2 + 2C)^k\|H\|^k \muqr(n)\mach$. For the third, observe that
    $$
        \|\ax{Q}^\ast H \ax{Q} - Q^*HQ\|_F \le \|\ax{Q}^\ast H (\ax Q - Q)\|_F + \|(\ax{Q}^\ast - Q^\ast)H Q\|_F \le  2\|H\|\|\ax{Q} - Q\|_F,
    $$
    and use the first assertion again.
\end{proof}

We close the subsection by giving forward error bounds for computing $\tau_p(H)^k = \|e_n^*p(H)^{-1}\|^{-1}$ indirectly, from the $R$'s output by $\iqr(H,p(z))$, for $p$ a polynomial of degree $k$.
\bigskip

\begin{boxedminipage}{\textwidth}
$$\comptau{k}$$
    \textbf{Input:} Hessenberg $H\in \bC^{n\times n}$, polynomial $p(z)=(z-s_1)\cdots (z-s_k)$ \\
    \textbf{Output:} $\nn \geq 0$ \\
    \textbf{Ensures:} $|\nn - \tau_p (H)^k| \le 0.001 \tau_p(H)^k$ 
\begin{enumerate}
    \item  $[\ax{\hat{H}},  \ax{R}_1, \dots, \ax{R}_k] \gets \iqr (H, p(z))$
    \item $\nn \gets \fl\left( (\ax{R}_1)_{nn}\cdots (\ax{R}_k)_{nn} \right)$
\end{enumerate} 
\end{boxedminipage}
\bigskip

\begin{lemma}[Guarantees for $\comptau{k}$] 
\label{lem:guaranteetaum}
    If $\calS = \{s_1,...,s_k\} \subset \mathbb{D}(0,C\|H\|)$ and
    \begin{align}
        \label{assum:comptau}
        \mach 
        &\le \mach_{\comptau{}}(n,k,C,\|H\|,\kappa_V(H),\dist(\calS,\Spec H)) \\
        &:= \frac{1}{6 \cdot 10^3 \kappa_V(H) \muqr(n)}\left(\frac{\dist(\calS,\Spec H)}{(2 + 2C)\|H\|}\right)^{2k} \\
        &= 2^{-O\left(\log n\kappa_V(H) + k\log \frac{\|H\|}{\dist(\calS,\Spec H)}\right)}, \nonumber
    \end{align}
    then $\comptau{k}$ satisfies its guarantees, and runs in $$T_{\comptau{}}(n,k):= T_{\iqr}(n, k) + k = O(k n^2)$$ arithmetic operations.
\end{lemma}

\begin{proof}
    Let $[Q, R] = \qr(p(H))$ and recall that  \eqref{eq:Rs} shows that $\tau_p(H)^k = R_{nn}$. As \eqref{assum:comptau} implies $\mach_{\iqr}(n,k,\|H\|,\kappa_V(H),\dist(\calS,\Spec H))$, we can apply Lemma \ref{lem:multiiqrstability}: the matrix $\ax{R}= \ax{R}_k\cdots \ax{R}_1$ satisfies
    \begin{align*}
     |\ax{R}_{n,n} - R_{n,n}|
     &\le \|\ax{R} - R\| \\
     &\le 12  \kappa_V(H)\left(\frac{(2 + 2C)^2\|H\|^2}{\dist(\calS,\Spec H)}\right)^k \muqr(n)\mach && \text{Lemma \ref{lem:multiiqrstability}} \\ 
     &\leq \frac{0.0005}{\|p(H)^{-1}\|}   && \text{by \eqref{assum:comptau} and $\|p(H)^{-1}\| \le \kappa_V(H)\dist(\calS,\Spec H)^{-k}$} \\
     &\le 0.0005 \, \sigma_{\min}(R) & & p(H) = QR \\
     &\le 0.0005 \, R_{n,n}. & & \sigma_{\min}(R) \le \|e_n^\ast R\| = R_{n,n}
    \end{align*}
     Now, because $\nn$ is the result of computing the product of the $(\ax{R}_i)_{n,n}$ in floating point arithmetic, we have $\Big|\nn - \ax{R}_{n,n}\Big| \le k \mach \ax{R}_{n,n}$, whence
     \begin{align*}
     \Big|\nn - R_{n,n}\Big| 
     &\le \Big|\nn - \ax{R}_{n,n}\Big| + \Big|\ax{R}_{n,n} - R_{n,n}\Big| \\
     &\le k\mach \ax{R}_{n,n} + 0.0005\,R_{n,n} \\
     &\le (1.0005k\mach + 0.0005)R_{n,n} \\
     &\le 0.001 R_{n,n}.
     \end{align*}
     It will also be useful to observe that
     $$
        \left|\frac{1}{\nn} - \frac{1}{R_{n,n}}\right| \le \frac{0.001}{|\nn|} \le \frac{0.001}{|\nn|} \le \frac{0.001}{\big|R_{n,n} - |\nn - R_{n,n}|\big|} \le \frac{0.001}{0.99 R_{n,n}} \le \frac{0.0011}{R_{n,n}}.
     $$
\end{proof}

\subsection{Shift Regularization}
\label{sec:regularization}

The forward error bounds on our shifts are controlled by the inverse of the distance to $\Spec H$; to ensure that this is not too large, we \emph{regularize} the shifts $r_1,\ldots,r_k$ by randomly perturbing them. 

\begin{lemma}[Regularization of shifts]
    \label{lem:fixguarantee1} 
    Let $\calR = \{r_1,...,r_k\}\subset \bC$  and  $\pretol\geq \tol >0$.  Assume  $$ \tol +\pretol \leq \frac{\gap(H)}{2}.$$ 
    Let $\w_1,...,\w_k \sim \text{Unif}(D(0, \pretol))$ be i.i.d. and $\chR = \{\chr_1,...,\chr_k\} = \{r_1+ \w_1,...,r_k + \w_k\}$. Then with probability at least $1-k\left(\tol/\pretol\right)^2$, we have $\dist(\chR,\Spec H) \ge \tol$.
\end{lemma}

\begin{proof}
    Define the bad region $\calB \subset \bC$ as the union of disks $\calB:= \bigcup_{\lambda \in \Spec(H)} D(\lambda, \eta_1)$. The assumption $\tol+\pretol \leq \gap(H)/2$ implies that for each $r_i$, the disk $D(r_i, \pretol)$ intersects at most one disk in $\calB$; since $\chr_i$ is distributed uniformly in $D(r_i, \pretol)$ we have
    $$
        \P[\chr_i \in \calB] \leq \left( \frac{ \tol}{\pretol}\right)^2,
    $$
    and the total probability that at least one $\chr_i$ lies in the bad region is at most $k$ times this by a union bound. 
    
\end{proof}

\section{Finding Forward Stable Optimal Ritz Values (or Decoupling Early)}
\label{sec:ritz}
The shifting strategy $\shkb$ in \cite{banks2021global} uses a specific notion of approximation for Ritz values, namely \textit{$\r$-optimality} as defined in \eqref{eqn:optdef}. In \cite{banks2021global} we assumed the existence of a black box algorithm for computing such optimal values. In this section we will show how to compute $\theta$-optimal Ritz values which are forward stable in the sense of Section \ref{sec:implicitQR} (or guarantee immediate decoupling).

The procedure consists of two steps, and relies on the black box algorithm $\smalleig$ for computing forward approximations of the eigenvalues of a $k\times k$ or smaller matrix, in the sense of Definition \ref{def:small-eig}. The first step of our approximation procedure is simply to compute forward approximations to the Ritz values using $\smalleig$. Second, we show the following dichotomy: for appropriately set parameters, any forward-approximate set of Ritz values $\calR$ of a Hessenberg matrix $H$ is either (i) $\r$-optimal {\em or} (ii) contains a Ritz value which can be used to decouple the matrix in a single degree $k$ implicit QR step (in fact, the proof shows that this Ritz value must be close to an eigenvalue of $H$, see Remark \ref{rem:ritzclose}). This is the content of Theorem \ref{lem:dichotomy}, which is established in Section \ref{sec:dichotomy}. We give a finite arithmetic implementation of this dichotomy in Section \ref{sec:sec:ritzordecouple}.

\subsection{The Dichotomy in Exact Arithmetic}
\label{sec:dichotomy}

In this subsection we show that for $\forward$ small enough and $\theta$ large enough, any set $\calR = \{r_1, \dots, r_k\}$ of $\forward$-forward approximate  Ritz values of $H$ either  yields a $\r$-optimal set of Ritz values, or one of the $r_i\in \calR$ has a small value of $\tau_{(z-r_i)^k}(H)$. 

\begin{theorem}[Dichotomy]
    \label{lem:dichotomy}
    Let $\Rho = \{\rho_1, \dots, \rho_k\}$ be the  Ritz values of $H$ and assume that $\calR=\{r_1, \dots, r_k\}$ satisfies  $|\rho_i-r_i|\leq \forward$ for all $i \in [k]$. If 
    \begin{equation}
        \label{eq:dichotomy-parameter-assumption}
        \theta \geq (2\kappa_V^4(H))^{1/2k} \quad\text{and}\quad \frac{\forward}{\gap(H)} \le  \frac{1}{2}\left(\frac{\r}{(2\kappa_V^4(H))^{1/2k}} - 1\right) =:\cc
    \end{equation}
    then at least one of the following is true:  
    \begin{enumerate}[label=\roman*)]
        \item \label{lem:dichotomy1} $\calR$ is a set of $\r$-optimal Ritz values of $H$. 
        \item \label{lem:dichotomy2} There is an $r_i\in \calR$  for which 
\begin{equation}\label{eqn:dichotomyresolvent} \|e_n^* (H-r_i)^{-k}\|^{1/k}\ge
\frac{1}{2\kappa_V(H)^{2/k}}\cdot\left(\frac{\pot(H)}{\|H\|+\beta}\right)\cdot\left( \frac{1-\frac{(2\kappa_V^4)^{1/2k}}{\theta}}{\forward}\right).
        \end{equation}
    \end{enumerate}
\end{theorem}

The remainder of this subsection is dedicated to the proof of Theorem \ref{lem:dichotomy}. Let $\Rho = \{\rho_1, \dots, \rho_k\}$ and $\calR=\{r_1, \dots, r_k\}$ be as in Lemma \ref{lem:dichotomy}, and set $\chi(z)=(z-\rho_1)\cdots (z-\rho_k)$ and $p(z)=(z-r_1)\cdots (z-r_k)$. Of course, by construction $\chi(z)$ is the characteristic polynomial of $\corner{H}{k}$. Our strategy in proving Theorem \ref{lem:dichotomy} will be to show that if \ref{lem:dichotomy1} does not hold, then \ref{lem:dichotomy2} does; assuming the former, we can get that
\begin{align}
    \E[|p(Z_H)|^2] 
    &\geq \frac{\|e_n^* p(H)\|^2}{\kappa_V(H)^2} && \text{Lemma \ref{lem:spectral-measure-apx}} \nonumber \\ 
    &\geq \frac{\r^{2k}\|e_n^* \chi(H)\|^2}{ \kappa_V(H)^2} && \text{Negation of \ref{lem:dichotomy1}} \nonumber \\ 
    &\geq \frac{\r^{2k} \E[|\chi(Z_H)|^2]}{ \kappa_V(H)^4} && \text{Lemma \ref{lem:spectral-measure-apx}} \\
    &= 2(1 + 2\cc)^{2k}\E[|\chi(Z_H)|^2] && \eqref{eq:dichotomy-parameter-assumption} \label{eq:ineqforpcheck} 
\end{align}
 
In other words, $\E[|p(Z_H)|^2]$ is much larger than $\E[|\chi(Z_H)|^2]$. On the other hand, by the assumptions in Theorem \ref{lem:dichotomy}, the roots of $p(z)$ and $\chi(z)$ are quite close. Intuitively, because $Z_H$ is supported on the eigenvalues of $H$, these two phenomena can only occur simultaneously if some root of $p(z)$ is close to an eigenvalue of of $H$ with significant mass under the distribution of $Z_H$. The following lemma, whose proof we will briefly defer, articulates this precisely. The lemma does not require any particular properties of $p$ and $\chi$ other than that their roots are close, so we will phrase it in terms of two generic polynomials $q$ and $\chq$; when we apply the lemma, we will set $q = \chi$ and $\chq = p$.

\begin{lemma}
    \label{lem:instabilityimpliessupport}
    Assume that $\frac{\beta}{\cc}\le \gap(H)$ with $\cc$ defined as in \eqref{eq:dichotomy-parameter-assumption}, $q(z):=(z-s_1)\cdots (z-s_k)$ for some $\calS = \{s_1, \dots, s_k\} \subset \disk(0,\|H\|)$, and let  $\chq(z):=(z-\chs_1)\cdots (z-\chs_k)$ with $\chs_1,...,\chs_k \in \mathbb{C}$ satisfying
    $$
        \max_{i\in [k]} |s_i-\chs_i| \leq \forward.
    $$ 
    Then
    $$
        \P\left[\Dist(Z_H,\{s_1,...,s_k\}) \le \frac{\forward}{2c}\right] \ge \frac{\E[|\chq(Z_H)|^2] - (1 + 2c)^{2k}\E[|q(Z_H)|^2]}{(2(\|H\| + \forward)(1 + 2c))^{2k}}.
    $$
\end{lemma}

Lemma in hand, we can now complete the proof.

\begin{proof}[Proof of Theorem \ref{lem:dichotomy}]
    Using Lemma \ref{lem:instabilityimpliessupport} with $q(z)=\chi(z)=(z-\rho_1)\ldots(z-\rho_k)$ and $\chq(z)=p(z)=(z-r_1)\ldots(z-r_k)$, we find that
    \begin{align*}
        \P\left[\Dist(Z_H,\Rho) \le \frac{\forward}{2c}\right]
        &\ge \frac{\E[|p(Z_H)|^2] - (1 + 2c)^{2k}\E[|\chi(Z_H)|^2]}{(2(\|H\| + \forward)(1 + 2c))^{2k}} \\
        &\ge \frac{\E[|\chi(Z_H)|^2]}{2^{2k}(\|H\| + \forward)^{2k}} & & \text{\eqref{eq:ineqforpcheck}} \\
        &\ge \frac{\|e_n^\ast \chi(H)\|^2}{2^{2k}\kappa_V(H)^2(\|H\| + \forward)^{2k}} & & \text{Lemma \ref{lem:spectral-measure-apx}} \\
        &= \frac{\pot^{2k}(H)}{2^{2k}\kappa_V(H)^2(\|H\| + \forward)^{2k}} & & \text{Lemma \ref{lem:minnorm}}.
    \end{align*}
    Since the right hand side is nonzero and $Z_H$ is supported on the spectrum of $H$ (and since $c \le 1/2$ by assumption) this implies that for some $i \in [k]$ and $\lambda \in \Spec(H)$
    $$|\rho_i-\lambda|\le \frac{\beta}{2c}.$$

    On the other hand, as we are assuming $\forward/\cc \le \gap(H)$, there can be at most one eigenvalue within $\forward/2\cc$ of each $\rho_i$ --- otherwise by the triangle inequality two such eigenvalues would be at distance less that $\forward/\cc \le \gap(H)$ from one another. Since there are only $k$ of the $\rho_i$'s, at least one of the eigenvalues, say $\lambda$, that is at least $\forward/2\cc$-close to of one of them must satisfy
\begin{equation}\label{eqn:dichotomyprob}
        \P\left[Z_H =\lambda\right] \ge \frac{1}{k}\left(\frac{\pot(H)}{2\kappa_V(H)^{1/k}(\|H\|+ \forward)}\right)^{2k}.
\end{equation}
By the triangle inequality, we then have
\begin{equation}\label{eqn:dichotomyclose}
        |r_i - \lambda| \le |r_i - \rho_i| + |\rho_i - \lambda| \le \forward\left(1 + \frac{1}{2c}\right).
\end{equation}
    
Finally,
        \begin{align*}
     \|e_n^* (H-r_i)^{-k}\|^{1/k}
        & \ge \frac{\E\left[|Z_H-r|^{-2k}\right]^{1/2k}}{\kappa_V(H)^{1/k}} & & \text{Lemma \ref{lem:spectral-measure-apx}} \\
        &\ge \frac{1}{\kappa_V(H)^{1/k}}\cdot \frac{1}{(2k)^{1/2k}}\left(\frac{\pot(H)}{\kappa_V(H)^{1/k}(\|H\|+ \forward)}\right)\cdot\left( \frac{2c}{(2c+1)\beta}\right),
    \end{align*} where the second inequality uses $\E\left[|Z_H-r_i|^{-2k}\right] \geq \frac{\P[Z_H=\lambda]}{|\lambda-r|^{2k}}$ and \eqref{eqn:dichotomyprob}, \eqref{eqn:dichotomyclose}. This yields the conclusion by substituting $c$ and noting that $(2k)^{1/2k}\le 2$.
\end{proof}

\begin{remark}\label{rem:ritzclose}
By   \eqref{eqn:dichotomyprob} and \eqref{eqn:dichotomyclose}, the above proof shows that the culprit Ritz value $r_i$ is close to an eigenvalue of $H$ and the corresponding right eigenvector has a large inner product with $e_n$. This could alternatively be used to decouple the matrix using other techniques such as inverse iteration.
\end{remark}
\begin{proof}[Proof of Lemma 
    \ref{lem:instabilityimpliessupport}]
    We begin by partitioning set $\calS = \{s_1,...,s_k\}$ according to which eigenvalue of $H$ is the closest: relabelling $\Spec(H) = \{\lambda_1,...,\lambda_n\}$ as necessary, write $\calS = S_1 \sqcup \cdots \sqcup S_\ell$, where $S_j$ consists of those $s_i$ whose closest eigenvalue is $\lambda_j$ (breaking ties arbitrarily).
    
    Now, recursively define a sequence of polynomials $q_0, \dots , q_l$ with $l\le k$ given by $q_0(z) = q(z)$ and
    $$
        q_{j+1}(z) := \frac{\prod_{i \in S_{j+1}}(z-\chs_i)}{\prod_{i\in S_{j+1}} (z-s_i)} q_{j}(z);
    $$
    in other words, the $q_j$ interpolate between $q$ and $\chq$ by exchanging the original roots $s_1,...,s_k$ for the perturbed ones $\chs_1,...,\chs_k$, doing so in batches according to the partition $\calS = S_1 \sqcup \cdots \sqcup S_\ell$. The proof reduces to the following bound on $\E[|q_{j}(Z_H)|^2]$ in terms of $\E[|q_{j-1}(Z_H)|^2]$, which we will prove shortly.
    
    \begin{claim}\label{cl:dichotomy}
        For each $j =1,...,\ell$, we have
        \begin{align*}
            \E[|q_{j}(Z_H)|^2] 
            &\le (1 + 2c)^{2|S_{j}|}\E[|q_{j-1}(Z_H)|^2] + (2(\|H\| + \forward))^{2k}\P[Z_H = \lambda_{j}] \mathbf{1}\left[\Dist(\lambda_{j},\calS) \le \tfrac{\forward}{2c}\right].
        \end{align*}
    \end{claim}
    
    \noindent In view of the claim, we can inductively assemble these bounds to compare $\E[|q(Z_H)|^2]$ and $\E[|\chq(Z_H)^2]$:
    \begin{align*}
        \E[|\chq(Z_H)|^2] 
        &= \E[|q_\ell(Z_H)|^2] \\
        &\le (1 + 2\cc)^{2|S_{\ell}|}\E[|q_{\ell-1}(Z_H)|^2] + (2(\|H\| + \forward))^{2k}\P[Z_H = \lambda_{\ell}] \mathbf{1}\left[\Dist(\lambda_{\ell},\calS) \le \tfrac{\forward}{2c}\right] \\
        &\le (1 + 2\cc)^{2k}\E[|q_0(Z_H)|^2] \\
        &\qquad + \sum_{i \in [\ell]}(2(\|H\|+ \forward))^{2k}(1 + 2\cc)^{2\sum_{j =1}^i|S_i|}\P[Z_H = \lambda_{i}]\mathbf{1}\left[\Dist(\lambda_i,\calS) \le \tfrac{\forward}{2\cc}\right] \\
        &\le (1 + 2\cc)^{2k}\left(\E[|q(Z_H)|^2] + (2(\|H\| + \forward))^{2k}\sum_{i \in [\ell]}\P[Z_H = \lambda_{i}]\mathbf{1}\left[\Dist(\lambda_i,\calS) \le \tfrac{\forward}{2\cc}\right]\right) \\
        &\le (1 + 2\cc)^{2k}\left(\E[|q(Z_H)|^2] + (2(\|H\| + \forward))^{2k}\P\left[\Dist(Z_H,\calS) \le \tfrac{\forward}{2\cc}\right]\right).
    \end{align*}
    Rearranging gives the bound advertised in the lemma.
    
    It remains to prove Claim \ref{cl:dichotomy}. To lighten notation, we'll write $s$ and $\chs$ for an arbitrary element in $S_j \subset \calS$, and its perturbation, respectively. For any $m\in [n]\setminus j$ and $s\in S_{j}$, we have $|\lambda_m-s|\geq \frac{\gap(H)}{2}$, so 
    $$
        \left|\frac{\lambda_m-\chs}{\lambda_m-s} \right| \leq 1+ \left| \frac{s-\chs}{\lambda_m-s} \right|
        \leq 1+\frac{2|s-\chs|}{\gap(H)} \leq 1+2\cc,
    $$
    and hence
    \begin{equation*}
        \label{eq:proportionalbound}
        \prod_{s\in S_j} \left|\frac{\lambda_m -\chs}{\lambda_m-s}\right| \leq ( 1+ 2 \cc )^{|S_{j}|}. 
    \end{equation*}
    Using the above, the definition of $q_j$ in terms of $q_{j-1}$, and expanding the expectation as a sum, we find 
    \begin{align}
        \E[|q_{j}(Z_H)|^2]
        &=  \P[Z_H=\lambda_{j}] |q_{j}(\lambda_{j})|^2 + \sum_{m \in [n]\setminus j} \P[Z_H=\lambda_m] |q_{j-1}(\lambda_m)|^2 \prod_{s \in S_{j+1}} \left| \frac{\lambda_m-\chs}{\lambda_m-s} \right|^2
        \nonumber \\ 
        &\le  \P[Z_H=\lambda_{j}] |q_{j}(\lambda_{j})|^2 + (1+2 \cc )^{2|S_{j}|}\, \sum_{m \in [n]\setminus j} \P[Z_H=\lambda_m] |q_{j-1}(\lambda_m)|^2 \nonumber   \\ 
        & \le \P[Z_H=\lambda_{j}] \left(  |q_{j}(\lambda_{j})|^2 - (1+2\cc)^{2|S_{j}|} |q_j(\lambda_{j-1})|^2\right) + (1+2 \cc )^{2|S_{j}|} \E[|q_{j-1}(Z_H)|^2] \label{eq:firstboundonqj} \\ 
        & \le \P[Z_H=\lambda_{j}] |q_{j-1}(\lambda_{j})|^2 \left( \prod_{s\in S_{j}}\left(1+ \left|\frac{s-\chs}{\lambda_{j}-s}\right|\right)^2- (1+2\cc)^{2|S_{j}|}  \right) \nonumber \\
        &\qquad\qquad +  (1+2 \cc )^{2|S_{j}|} \E[|q_{j-1}(Z_H)|^2] \label{eq:secondbounfonqj} 
    \end{align}
    We have defined $S_j$ so that $\lambda_j$ is the closest eigenvalue to every $s \in S_j$, so $\dist(\lambda_j,\calS) = \dist(\lambda_j, S_j)$. Thus when $\dist(\lambda_j,\calS) > \tfrac{\forward}{2c}$, we can rearrange to see that
    \begin{align*}
        0 &\ge \left(1 + \frac{\forward}{\dist(\lambda_j,S_j)}\right)^{2|S_j|} - (1 + 2c)^{2|S_j|} \\
        &\ge \prod_{s \in S_j}\left(1 + \frac{|s - \check s|}{|\lambda_j - s|}\right)^{2} - (1 + 2c)^{2|S_j|};
    \end{align*}
    the latter is a factor of the first term on the right hand side of \eqref{eq:secondbounfonqj}, so in the event $\dist(\lambda_j,\calS) > \tfrac{\forward}{2c}$ we have
    $$
        \E[|q_j(Z_H)|^2] \le (1 + 2c)^{2|S_j|}\E[|q_{j-1}(Z_H)|^2].
    $$
    On the other hand, independent of $\dist(\lambda_j,\calS)$ (and thus in particular when $\dist(\lambda_j,\calS) \le \tfrac{\forward}{2c}$)  from \eqref{eq:firstboundonqj} we know that
    \begin{align*}
        \E[|q_{j}(Z_H)|^2]
        &\leq \P[Z_H=\lambda_{j}]|q_{j}(\lambda_{j})|^2 + (1+2 \cc )^{2|S_{j}|} \E[|q_{j-1}(Z_H)|^2] \\
        &\leq \P[Z_H=\lambda_{j}](2(\|H\| + \forward))^{2k} + (1+2 \cc )^{2|S_{j}|} \E[|q_{j-1}(Z_H)|^2.
    \end{align*}
    For the final inequality, note that $\lambda_j\in \disk(0,\|H\|)$ and, because $\calS \subset \disk(0,\|H\|)$, and $|\check s - s| \le \forward$ for every $s \in \calS$, the roots of each $q_j$ are contained in $\disk(0,\|H\| + \forward)$. Combining the bounds on $\E[|q_j(Z_H)|^2]$ in the cases $\dist(\lambda_j,\calS) > \tfrac{\forward}{2c}$ and $\dist(\lambda_j,\calS) \le \tfrac{\forward}{2c}$, we find that
    \begin{align*}
        \E[|q_{j}(Z_H)|^2] 
        &\le (1 + 2\cc)^{2|S_{j}|}\E[|q_{j-1}(Z_H)|^2] + (2(\|H\| + \forward))^{2k}\P[Z_H = \lambda_{j}] \mathbf{1}\left[\Dist(\lambda_{j},\calS) \le \tfrac{\forward}{2c}\right],
    \end{align*}
    establishing the claim.
\end{proof}

\subsection{Finite Arithmetic Implementation of $\regularize$}
\label{sec:sec:ritzordecouple}
In this subsection we combine Theorem \ref{lem:dichotomy} and the regularization procedure of Lemma \ref{lem:fixguarantee1} to obtain a finite arithmetic algorithm $\regularize$ for finding $\r$-optimal Ritz values in the sense of \cite[Definition 1.2]{banks2021global}, for $\r$ set as in \eqref{eqn:settheta}, and with the additional property of being forward stable. The first step is testing whether a set of putative approximate Ritz values are $\theta$-optimal.\\

\noindent
\begin{boxedminipage}{\textwidth}
    $$\optimal$$
    \textbf{Input:} Hessenberg $H\in \bC^{n\times n}$, $\{s_1, \dots, s_k\}= \calS \subset \bC$ \\
    \textbf{Global Data:} Optimality parameter $\r$ \\
    \textbf{Output:} Optimality flag $\optflag$ \\
    \textbf{Ensures:} If $\optflag = \true$, then $\calS$ are $\r$-optimal; if $\optflag = \false$, then they are not $(.998^{1/k}\r)$-optimal
    \begin{enumerate}
        \item  $\ax{v_0} \gets e_n $
        \item \textbf{For} $j=0, \dots, k-1 $
        \begin{enumerate}
            \item $\ax{v_{j+1}} \gets \fl\left( (H-s_{j+1})^*\ax{v_j}\right)$
        \end{enumerate}
        \item \textbf{If} $\fl(\|\ax{v}_k\|) \ge .999\r^k \pot^k(H)$, $\optflag \gets \false$, \textbf{else} $\optflag \gets \true$
    \end{enumerate} 
\end{boxedminipage}

\begin{lemma}[Guarantees for $\optimal$]
    \label{lem:optimal-guarantees}
    Assume that $s_1, \dots, s_k \in \disk(0, C\|H\|)$ and
    \begin{equation}
        \label{assump:optimal}
        \mach \le \mach_{\optimal}(n,k,C,\|H\|,\r) := \frac{1}{2\cdot 10^3 n^2}\left(\frac{\pot(H)}{\r(2 + 2C)\|H\|}\right)^k = 2^{-O\left(\log n + k\log \tfrac{\theta\|H\|}{\pot(H)}\right)};
    \end{equation}
    then $\optimal$ satisfies its guarantees and runs in at most $T_{\optimal}(k) := 4k^2 = O(k^2)$ arithmetic operations.
\end{lemma}
\begin{proof}[Proof of Lemma \ref{lem:optimal-guarantees}]
    From our initial floating point assumptions, we have $\ax{v_i} = (H - s_i)\ax{v_{i-1}} + \Delta_i$, where $\Delta$ is supported only on its $i+1$ final coordinates, each of which has magnitude at most $(1 + C)\|H\|\|\ax{v_{i-1}}\| \cdot n\mach$, giving the crude bound $\|\Delta_i\| \le (1 + C)\|H\|\|\ax{v_{i-1}}\| \cdot n^{3/2}\mach$. Thus inductively
    $$
        \|\ax{v_i}\| \le \left((1 + C)\|H\|(1 + n^{3/2}\mach)\right)^i
    $$
    and given $\mach \le n^{-3/2}$,
    \begin{align*}
        \left|\fl\left(\|\ax{v_k}\|\right) - \|e_n^\ast p(H)\|\right| 
        &\le n\mach\|\ax{v_k}\| +  \left|\|\ax{v_k}\| - \|e_n^\ast p(H)\|\right| \\
        &\le n\mach\left((1 + C)\|H\|(1 + n^{3/2}\mach)\right)^k + kn^{3/2}\mach \cdot \left((1 + C)\|H\|(1 + n^{3/2}\mach)\right)^k \\
        &\le 2n^2(2 + 2C)^k\|H\|^k\mach.
    \end{align*}
    Thus if $\fl(\|\ax{v_k}\|) \ge .999\r^k\pot^k(H)$, our assumption on $\mach$ ensures
    $$
        \|e_n^\ast p(H)\| \ge .999\r^k\pot^k(H) - 2(1 + C)^k\|H\|^k k^2 n^{3/2}\mach \ge .998\r^k\pot^k(H).
    $$
    On the other hand, if $\fl(\|\ax{v_k}\|) \le .999\r^k\pot^k(H)$, then analogously we have
    $$
        \|e_n^\ast p(H)\| \le \r^k\pot^k(H).
    $$
    
    For the running time, each $\ax{v_i}$ is supported only on $i+2$ coordinates, so each multiplication $(H - s_i)\ax{v_{i-1}}$ requires $3i + 3$ arithmetic operations, for a total of $3k(k+1)/2$; we then require a further $2k$ to compute $\|\ax{v_k}\|$, giving $3k(k+1)/2 + 2k \le 4k^2$ arithmetic operations overall.
\end{proof}
We now specify $\regularize$ in full. \\

\noindent 
\begin{boxedminipage}{\textwidth}
    $$\regularize$$
    \textbf{Input:} Hessenberg $H$, working accuracy $\wacc$, failure probability $\phi$ \\
    \textbf{Global Data:} Norm bound $\scale$, optimality parameter $\r$ as in \eqref{eqn:settheta} \\
    \textbf{Requires:} $H$ is $\wacc$-unreduced, $\|H\| \le \scale$, $\gap(H)\ge \frac{2\wacc^2}{\scale}$, $k/\phi \geq 2 $\\
    \textbf{Output:} Hessenberg $\next H$, $\r$-approximate Ritz values $\chR$, decoupling flag $\decouple$ \\
    \textbf{Ensures:} With probability at least $1 - \phi$, $\dist(\chR,\Spec H) \ge \tol$ (as defined in line 1) \textit{and} one of the following holds:
    \begin{itemize}
        \item $\decouple = \false$, $\next H = H$, and $\chR$ is an exact set of $\r$-optimal Ritz values of $H$, satisfying $\chR \subset \disk(0,1.1\|H\|)$.
        \item $\decouple = \true$  and for some $\chr \in \chR$, $\next H = \iqr(H,(z - \chr)^k)$ is $\wacc$-decoupled.
    \end{itemize}
    \begin{enumerate}
        \item $\beta\gets \frac{\wacc^2}{16\cdot 101\cdot \scale}, \pretol\gets \frac{\beta}{2}, \tol\gets \frac{\pretol}{\sqrt{2k/\phi}}=\frac{\wacc^2\sqrt{\phi}}{32\cdot 101\cdot \scale\sqrt{2k}}$
        \item $\calR \gets \smalleig\left(\corner{H}{k},\beta/2,\phi/2\right)$
        \item $\{\chr_1, \dots, \chr_k\}=\chR\gets \{r_1+\w_1, \dots, r_k+ \w_k\}$, where the $\w_i$ are i.i.d samples from $\unif\big(\disk(0, \pretol )\big)$
        \item \textbf{If} $\optimal(\chR,H,\r) = \true$, set $\next H \gets H$ and $\decouple \gets \false$
        \item \label{stage:testingoptimality} \textbf{Else if} $\optimal(\chR,H,\theta) = \false$, \textbf{for} $i=1,...,k$
        \begin{enumerate}
            \item $\next H \gets \iqr(H,(z-\chr_i)^k)$
            \item \textbf{If} $\next H_{j+1,j} \le \wacc$ for any $j \in \{n-k, n-k+1, \dots, n-1\}$, set $\decouple \gets \true$ and halt
        \end{enumerate}
    \end{enumerate}
\end{boxedminipage}

\begin{lemma}[Guarantees for $\regularize$] \label{lem:regularize-guarantees}
    Assuming that
    \begin{align}
        \mach &\le \mach_{\regularize}(n,k,\scale,\K,\r,\wacc, \phi) \nonumber \\
        &:= \min\left\{\mach_{\optimal}(n,k,1.1,\scale,\r), \frac{\wacc}{8n^{1/2} \scale} \mach_{\iqr}\left(n,k,1.1,\scale,\K,\frac{\wacc^2\sqrt{\phi}}{32\cdot 101 \cdot \scale\sqrt{2k}}\right)\right\} \label{assump:regularize} \\
        &= 2^{-O\left(\log n\K + k \log \frac{\r\|H\| \cdot k \scale}{\wacc\phi}\right)}
    \end{align}
    then $\regularize$ satisfies its guarantees and its running time depends on the value of the decoupling flag.  In either case it makes one call to $\smalleig$, in addition to that call
    \begin{enumerate}
        \item if $\decouple = \false$, $\regularize$ uses at most
        $$T_{\regularize}(n, k, \false) := k \cd+k+  T_{\optimal}(k) = O(k^2)$$
        arithmetic operations. 
        \item otherwise, $\regularize$ uses at most 
        $$ T_{\regularize}(n, k, \true) :=  T_{\optimal}(k) + k(T_{\iqr}(n,k)+ k+\cd+1) = O(k^2n^2) $$
         arithmetic operations.
    \end{enumerate}
\end{lemma}

\begin{proof}
    First, the assumptions of $\regularize$ on its input parameters imply that
    $$
       \tol+\pretol\le \beta\le\frac{\wacc^2}{\scale}\le \gap(H)/2
    $$
    so we can apply Lemma \ref{lem:fixguarantee1} to find that $\dist(\chR,\Spec H) \ge \tol$ with probability at least 
    $$1 - k\left(\frac{\tol}{\pretol}\right)^2 = 1 - k\left(\sqrt{\frac{\phi}{2k}}\right)^2 \ge 1 - \phi/2.
    $$

    By the black box assumptions on $\smalleig$, $\calR$ is a set of $\beta/2$-forward approximate Ritz values with probability at least $1-\phi/2$. The perturbed set $\chR$ are in this case $\beta$-forward approximate Ritz values, and we further have
    $$
        \beta \le 0.1\wacc \le 0.1\|H\|
    $$
    so the set $\chR$ is contained in a disk of radius $1.1\|H\|$.
    
    The assumption $\mach \le \mach_{\optimal}(1.1,k,n,H)$ means that if $\optimal(\chR,H,\r) = \true$ we are guaranteed that $\chR$ is indeed a set of $\r$-optimal Ritz values for $H$. On the other hand if $\optimal(\chR,H,\r) = \false$, then by Lemma \ref{lem:optimal-guarantees} the $\chR$ fail to be $0.998^{1/k}\r$-optimal. Examining the definitions of $\theta$ and $\beta$, we verify the hypotheses of Theorem \ref{lem:dichotomy}:
    $$c= \frac{1}{2}\left(\frac{0.998^{1/k}\r}{(2\kappa_V(H)^{4})^{1/2k}}-1\right)\ge \frac{1}{2}\left(\frac{\frac{101}{100}(2B^4)^{1/2k}}{(2B^{4})^{1/2k}}-1\right)=\frac{1}{200}\ge \frac{\beta}{\gap(H)},$$
    and conclude that there is some $\chr \in \chR$  for which
    \begin{align*}
        \|e_n^* (H-\chr)^{-k}\|^{1/k} &\ge
        \frac{1}{2\kappa_V(H)^{2/k}}\cdot\left(\frac{\pot(H)}{\|H\|+\beta}\right)\cdot\left( \frac{1-\frac{(2\kappa_V^4)^{1/2k}}{0.998^{1/k}\r}}{\forward}\right)\\
        &\ge \frac{1}{4}\cdot\left(\frac{\wacc}{2\scale}\right)\cdot\left( \frac{1-\frac{100}{101}}{\forward}\right) & &\text{$B^{2/k}\le 2$, $\psi_k(H)\le\wacc$, $\beta\le \|H\|$}\\
        &\ge \frac{2}{\wacc}
        \end{align*}  
    by the definition of $\beta$ in line $1$. In the event that $\dist(\chR,\Spec H) \ge \tol$, our choice of $\mach$ in \eqref{assump:regularize} means that we can apply Lemma \ref{lem:multiiqrstability} to $\next{H}=\iqr(H,(z-\chr)^k)$ with $C=1.1$, giving
    $$
        \|\next{H}-\exactqr(H,(z-\chr)^k)\|_F\le 32 \kappa_V(H)\|H\| \left(\frac{4.2 \|H\|}{\dist(\chr,\Spec H)}\right)^k n^{1/2}\muqr(n)\mach \le \wacc/2.
    $$
    Using $\psi_k(\exactqr(H,(z-\chr)^k)\le\tau_{(z-\chr)^k}(H)\le \wacc/2$ (which was verified in \cite[Lemma 2.3]{banks2021global}) we find that $\exactqr(H,(z-\chr)^k)$ has a subdiagonal entry smaller than $\wacc/2$, so $\next{H}$ must have a subdiagonal entry smaller than $\wacc$, completing the proof of correctness.
    
    To analyze the running time,  note that  when $\decouple=\false$ other than the call to $\smalleig$, in line 3  $k$ samples are taken from $\unif(\disk(0, \pretol))$ and $k$ additions are made which amounts to $\cd k+k$ operations, and in line 4 $\optimal$ is called once, adding $T_{\optimal}(k)$ to the running time.  In addition to that, when $\decouple=\true$, at most $k$ calls to $\iqr$ with degree $k$ are made and each time $k$ subdiagonals of $\next{H}$ are checked, adding $kT_{\iqr}(n, k)+k^2$ operations. 
\end{proof}

\section{Finite Arithmetic Analysis of One Iteration of $\Sh_{k,\K}$}
\label{sec:shift}

In this section we provide the finite arithmetic analysis of a single iteration of the shifting strategy $\Sh_{k,\K}$ introduced in \cite{banks2021global}; we assume familiarity with the context and notions introduced there. In exact arithmetic, $\Sh_{k,\K}$ takes as input a Hessenberg matrix $H$ with $\kappa_V(H) \le \K$, and a set $\calR$ of $\r$-optimal Ritz values for $H$, and ouputs a new Hessenberg matrix $\hat H$ unitarily equivalent to $H$, with $\pot(\hat H) \le (1 -\gamma)\pot(H)$. Along the way, it first uses a subroutine $\find$ to generate a promising Ritz value $r \in \calR$ and then --- in the event that the shift $(z - r)^k$ does not reduce the potential --- uses a subroutine $\exc$ to produce a set of exceptional shifts $\calS$, one of which is guaranteed to achieve potential reduction. Let us now specify these subroutines in finite arithmetic and state their guarantees.

\paragraph{Computation of $\tau$ and $\psi$.} The shifting strategy $\Sh_{k,B}$ needs access to both $\tau_p(H)$ and $\pot(H)$. The former can be computed using Lemma \ref{lem:guaranteetaum}. For the latter, we will assume for simplicity that $\pot^k(H)$ (which is simply a product of $k$ entries of $H$) can be computed \textit{exactly} (this could for instance be achieved by temporary use of moderately increased precision). On the other hand, in some places it will be important to account for the error in computing the $k$-th root of $\pot^k(H)$, so we will denote 
   $$ \ax{\pot}(H) := \fl\left(\left(\pot^k(H)\right)^{1/k}\right),$$
 and assume 
 \begin{equation}
 \label{eq:boundonpottilde}
    |\ax{\pot}(H) - \pot(H)| \le (1 - 0.999^{1/k})\pot(H) \le 0.001 \pot(H),
 \end{equation}
which as per Lemma \ref{lem:elementaryfunctions} can be computed in $T_{\psi}(k) := k + T_{\rt}(k,1 - 0.999^{1/k})$ arithmetic operations provided that 
\begin{equation}
    \label{eq:mach-requirement-pot}
    \mach \le \mach_\psi(k) := \frac{1 - 0.999^{1/k}}{k(\croot + 1 - 0.999^{1/k})} = 2^{-O\left(\log k\right)}. 
\end{equation} 
This setting of the accuracy of $\ax\pot$ will be convenient for the analysis of $\exc$ below.

\paragraph{Analysis of $\find$.} To produce a promising Ritz value with $\find$, we will proceed as in the exact arithmetic case, using $\comptau{k}$ to guide our binary search procedure. The guarantees on $\comptau{k}$ are only strong enough to ensure that we discover a $(1.01\kappa_V(H))^{\frac{4\log k}{k}}$-promising Ritz value --- as opposed the $\kappa_V(H)^{\frac{4\log k}{k}}$-optimality we are guaranteed in the exact case.\\

\noindent \begin{boxedminipage}{\textwidth}
    $$\find$$
    \textbf{Input:} Hessenberg $H$, a set $\calR=\{r_1,\ldots,r_k\} \subset \bC$ \\
    \textbf{Global Data:} Promising parameter $\cp = (1.01\K)^{\frac{4\log k}{k}}$ as in \eqref{eqn:settheta} \\
    \textbf{Output:} A complex number $r\in \calR$ \\
    \textbf{Requires:} $\psi_k(H)>0$ \\
    \textbf{Ensures:} $r$ is $\cp$-promising  \\
    \begin{enumerate}
        \item \textbf{For} $j = 1,...,\log k$
        \begin{enumerate}
            \item Evenly partition $\calR = \calR_0 \sqcup \calR_1$, and \textbf{for} $b = 0,1$ set $p_{j,b} = \prod_{r \in \calR_{b}}(z - r)$
            \item $\calR \gets \calR_{\ax{b_j}}$, where $\ax{b_j}$ is the $b$ that minimizes $\comptau{k/2}(H,p_{j,b}^{2^{j-1}})$
        \end{enumerate}
        \item Output $\calR = \{r\}$
    \end{enumerate}
\end{boxedminipage}

\begin{lemma}[Guarantees for $\find$]
    \label{lem:find-guarantees}
    Assume that $\calR \subset \disk(0,C\|H\|)$ and
    \begin{align}
        \label{eq:mach-requirement-find}
        \mach 
        &\le \mach_{\find}(n,k,C,\|H\|,\kappa_V(H),\dist(\calR,\Spec H)) \nonumber \\
        &:= \mach_{\comptau{}}(n,k/2,C,\|H\|,\kappa_V(H),\dist(\calR,\Spec H)) \\
        &= 2^{-O\left(\log n\kappa_V(H) + k\log \frac{\|H\|}{\dist(\calR,\Spec H)}\right)} \nonumber.
    \end{align}
    Then $\find$ satisfies its guarantees, and runs in
    $$
        T_{\find}(n,k):= 2\log k T_{\comptau{}}(n,k/2) + \log k = O(k\log k \cdot  n^2)
    $$ 
    arithmetic operations.
\end{lemma}

\begin{proof}
    The definition of $\mach_{\find}$ is sufficient to let us invoke Lemma \ref{lem:guaranteetaum} and conclude that it satisfies its guarantees throughout $\find$. On each step of the iteration, write $b_j$ for the $b \in \{0,1\}$ maximizing $\|e_n^\ast p_{j,b}(H)^{-1}\|$. Applying Lemma \ref{lem:guaranteetaum}, for each $b \in  \{0,1\}$ we have
    $$
        \left|\comptau{k/2}(H,p_{j,b}) - \|e_n^\ast p_{j,b}(H)^{-1}\|^{-1}\right| \le 0.0011\|e_n^\ast p_{j,b}(H)^{-1}\|^{-1},
    $$
    and thus it always holds that
    $$
        \|e_n^\ast p_{j,\ax{b_j}}(H)^{-1}\|^2 \ge (1 - 0.0022)^2 \|p_{j,b_j}(H)^{-1}\|^2 \ge \frac{1}{2.02}\left( \|p_{j,0}(H)^{-1}\|^2 +  \|p_{j,1}(H)^{-1}\|^2\right).
    $$
    We now mirror the proof of the analogous Lemma 2.7 in Part 1 of this work, which analyzes $\find$ in exact arithmetic. On each step of the iteration, we have defined thing so that
    \begin{equation}
        \label{eq:pj-identity}
        p_{j,\ax{b_j}}(z) = p_{j+1,0}(z) p_{j+1,1}(z).
    \end{equation}
    On the first step of the subroutine, this identity becomes $p(z) = p_{1,0}(z)p_{1,1}(z)$, where $p(z)$ is the polynomial whose roots are the full set $\calR$ of approximate Ritz values, so 
    \begin{align*}
        \|e_n^\ast p_{1,\ax{b_1}}(H)^{-1}\|^2
        &\ge \frac{1}{2.02}\left(\|e_n^\ast p_{1,0}(H)^{-1}\|^2 + \|e_n^\ast p_{1,1}(H)^{-1}\|^2\right) \\
        &\ge\frac{1}{1.01\kappa_V(H)^2}\E\left[\frac{1}{2}\left( |p_{1,0}(Z_H)|^{-2} + |p_{1,1}(Z_H)|^{-2}\right)\right] & & \text{Lemma \ref{lem:spectral-measure-apx}} \\
        &\ge \frac{1}{1.01\kappa_V(H)^2}\E[|p(Z_H)|^{-1}] & & \text{AM/GM and \eqref{eq:pj-identity}}   
    \end{align*}
    Applying the same argument to each subsequent step,
    \begin{align*}
        \|e_n^\ast p_{j+1,\ax{b_{j+1}}}(H)^{-2^{j}}\|^2 
        &\ge \frac{1}{1.01\kappa_V(H)^2}\E\left[\frac{1}{2}\left(|p_{j+1,0}(Z_H)|^{-2^{j+1}} + |p_{j+1,1}(Z_H)|^{-2^{j+1}}\right)\right] & & \text{Lemma \ref{lem:spectral-measure-apx}} \\
        &\ge \frac{1}{1.01\kappa_V(H)^2}\E\left[|p_{j+1,0}(Z_H) p_{j+1,1}(Z_H)|^{-2^{j}}\right] & & \text{AM/GM} \\
        &\ge \frac{1}{1.01\kappa_V(H)^4} \|e_n^\ast (p_{j+1,0}(H)p_{j+1,1}(H))^{-2^{j-1}}\| & & \text{Lemma \ref{lem:spectral-measure-apx}} \\
        &= \frac{1}{1.01\kappa_V(H)^4} \|e_n^\ast p_{j,\ax{b_j}}(H)^{-2^{j-1}}\|. & & \text{\eqref{eq:pj-identity}}
    \end{align*}
    Paying a further $\kappa_V(H)^2$ on the final step to convert the norm into an expectation, we get
    $$
        \E\left[|Z_H - r|^{-k}\right] \ge \left(\frac{1}{1.01\kappa_V(H)}\right)^{4\log k}\E\left[|p(Z_H)|^{-1}\right]
    $$
    as promised.
    
    For the runtime, we make $2\log k$ calls to $\comptau{k/2}$ and $\log k$ comparisons of two floating point numbers.
\end{proof}

\paragraph{Analysis of $\exc$.} We now come to the exceptional shift, effectuated by the subroutine $\exc$ in the event that a promising Ritz value fails to achieve potential reduction. In finite arithmetic, we will again proceed similarly to the exact arithmetic setting --- however, we will additionally need to ensure that all of our exceptional shifts are {forward stable} in the sense of Section \ref{sec:implicitQR}, and to achieve this we will apply a random perturbation in the same spirit as Section \ref{sec:regularization}. 

Let us first pause to prove a key lemma ensuring potential reduction in finite arithmetic for sufficiently forward stable shifts. In particular, we will use the forward error guarantee of Lemma \ref{lem:multiiqrstability} to analyze the potential of $\iqr(H,p(z))$, by directly comparing it to that of $\exactqr(H,p(z))$. 
\begin{lemma}
    \label{lem:multi-iqr-potential-apx}
    Let $p(z) = (z - s_1)...(z - s_m)$ for some floating point complex numbers $\calS = \{s_1,...,s_m\} \subset \disk(0,C\|H\|)$, and assume that for some $\wacc>0$,
    \begin{align}
        \label{eq:mach-requirement-potential-apx}
        \mach &\le \mach_{\ref{lem:multi-iqr-potential-apx}}(n,k,C,\|H\|,\kappa_V(H),\dist(\calS,\Spec H),\wacc) \nonumber \\
        &:= 0.001\wacc \cdot \frac{ \dist(\calS,\Spec H)^k}{32 \kappa_V(H) \|H\|^{k+1}(2 + 2C)^k n^{1/2}\muqr(n)} \\
        &= 
        2^{-O\left(\log \frac{n\kappa_V(H)}{\wacc} + k \log\frac{\|H\|}{\dist(\calS,\Spec H)}\right)}. \nonumber
    \end{align}
    Then at least one of the following holds:
    \begin{enumerate}
        \item ($\wacc$-Decoupling) Some subdiagonal of $\iqr(H,p(z))$ is smaller than $\wacc$.
        \item (Potential Approximation) $\pot(\iqr(H,p(z))) \le 1.0011 \pot(\exactqr(H,p(z)))$.
    \end{enumerate}
\end{lemma}
\begin{proof}
    Calling $\ax{\next{H}} = \iqr(H,p(z))$ and $\next{H} = \exactqr(H,p(z))$, one of two cases are possible. If $\next{H}_{i+1,i} < 0.999\wacc$ for some $i \in [n-1]$, then applying Lemma \ref{lem:multiiqrstability} and our assumption on $\mach$, 
    $$
        \ax{\next{H}}_{i+1,i} < \next{H}_{i+1,i} + 0.001\wacc < \wacc.
    $$
    On the other hand, if for every $i \in [n-1]$ we have $\next{H}_{i+1,i} \ge 0.999\wacc$, then
    $$
        \pot\Big(\ax{\next{H}}\Big) \le \pot\big(\next{H}\big)\left(\prod_{i \in [n-1]}\left(1 + \frac{0.001\wacc}{\next{H}_{i+1,i}}\right)\right)^{1/k} \le 1.0011\pot\big(\next{H}\big).
    $$
\end{proof}
\newcommand{\stagrat}{\xi}
\noindent \begin{boxedminipage}{\textwidth}
    $$\exc$$
    \textbf{Input:} Hessenberg $H$, initial shift $r$, working accuracy $\wacc$, stagnation ratio $\stagrat$, failure probability tolerance $\phi$ \\
    \textbf{Global Data:} Condition number bound $\K$, decoupling rate $\gamma$, norm bound $\scale$, optimality parameter $\r$, promising parameter $\cp$ \\
    \textbf{Output:} Finite subset $\calS\subset \bC$. \\
    \textbf{Requires:} $H$ is $\wacc$-unreduced, $\kappa_V(H) \le \K$, $\|H\| \le \scale$,
    $r$ is a $\theta$-approximate, $\cp$-promising Ritz value, and $\tau_{(z-r)^k}(H) \ge \stagrat \pot(H)$\\
    \textbf{Ensures:} With probability at least $1 - \phi$, some $s \in \calS$ satisfies at least one of
    \begin{itemize}
        \item ($\wacc$-Decoupling) A subdiagonal of $\iqr(H,(z-s)^k)$ is smaller than $\wacc$
        \item (Potential Reduction) $\pot(\iqr(H,(z-s)^k)) \le 1.0011(1-\gamma)\pot(H)$
    \end{itemize}
    \begin{enumerate}
        \item \label{line:exc1} $\ax R \gets 2^{1/k} \alpha B^{1/k}\r \ax{\pot}(H)$
        \item \label{line:exc2} $\varepsilon \gets \left (\frac{\stagrat(1 - \gamma)}{(13 \K^4)^{1/k} \cp^2\r^2}\right)^{\frac{k}{k-1}}$
        \item $\calS_0 \gets$ maximal $0.99\varepsilon$-net of $\disk\big(0,1 + \varepsilon\big)$
        \item $w \sim \unif \left(\disk\big(0,\varepsilon\ax R\big)\right)$
        \item $\calS \gets \fl\left((r + w + \ax R\calS_0)\right) \cap \disk\big(r,\ax R\big)$
    \end{enumerate}
\end{boxedminipage}

\begin{lemma}[Guarantees for $\exc$]
    \label{lem:exc-guarantees}
    Assume that $|r| + 1.001\r\cp\K^{1/k}\pot(H) \le C\|H\|$ and
    \begin{align}
        \mach &\le \mach_{\exc}(n,k,C,\scale,\K,\r,\wacc,\phi,\gamma,\stagrat,\cp) \nonumber \\
        &:= \min\left\{\mach_\psi(k), \frac{0.1 \varepsilon \cdot 1.998\r\cp\K\wacc}{4(\varepsilon + 2(1 + \varepsilon)C\scale)}, \right. \nonumber \\
        &\qquad \left. \mach_{\ref{lem:multi-iqr-potential-apx}}\left(n,k,C,\scale,\K,\left(\frac{\stagrat(1 - \gamma)}{(13 \K^4)^{1/k} \cp^2\r^2}\right)^{\frac{k}{k-1}} \cdot \frac{1.998\, \r \cp \K^{1/k}\wacc\sqrt{\phi}}{\sqrt{3n}},\wacc\right)\right\}  \label{eq:mach-requirement-exc} \\
        &= 2^{-O\left(k \log \frac{n\scale \K \cp\r}{\stagrat(1 - \gamma)\wacc\phi}\right)}.
    \end{align}
    Then $\exc$ satisfies its guarantees and runs in at most
    $$
        T_{\exc}(n,k,\stagrat,\gamma,\K,\cp,\r):= T_{\psi}(k) + 2S\left(\left(\frac{\stagrat(1 - \gamma)}{(13 \K^4)^{1/k} \cp^2\r^2}\right)^{\frac{k}{k-1}}\right) + \cd + O(1) = O\left(\K^{\frac{8}{k-1}}\left(\frac{\cp^2\r^2}{\stagrat(1 - \gamma)}\right)^{\frac{2k}{k-1}}\right)
    $$
    arithmetic operations and
    $$
        |\calS| \le S\left(\left(\frac{\stagrat(1 - \gamma)}{(13 \K^4)^{1/k} \cp^2\r^2}\right)^{\frac{k}{k-1}}\right) = O\left(\K^{\frac{8}{k-1}}\left(\frac{\cp^2\r^2}{\stagrat(1 - \gamma)}\right)^{\frac{2k}{k-1}}\right)
    $$
    where the function $S(\varepsilon) = O(\varepsilon^{-2})$ is defined in \eqref{eq:S-def}.
\end{lemma}

\begin{proof}
    From (\ref{eq:boundonpottilde}), the fact that $\mach \le \mach_{\psi}(k)$ we can bound
    \begin{equation}
    \label{eq:boundsonR}
      1.998\, \r\alpha \K \pot(H)  \leq (2\cdot .999)^{1/k} \r\alpha \K \pot(H) 
        \le \ax R 
        \le 1.001\cdot  \r\alpha \K^{1/k}\pot(H),
    \end{equation}
    meaning that (as $\pot(H) \le \|H\|$) the set $\calS$ is contained in a disk of radius $|r| + 1.001\r\cp\K^{1/k}\|H\|=C\|H\|$. We can then obtain that 
    \begin{align*}
        \P\left[Z_H \in \disk(r,\ax R)\right]
        &\ge \P\left[|Z_H - r| \le  1.998\, \r\alpha \kappa_V^{1/k}(H)\pot(H)\right] && \text{by (\ref{eq:boundsonR})}\\
        &\ge \left(1 - \frac{1}{1.998}\right)^2\frac{\stagrat^{2k}}{\kappa_V(H)^4\cp^{2k}\r^{2k}} && \text{\cite[Lemma 2.8]{banks2021global} with }t= \frac{1}{1.998}  \\
        &\ge \frac{0.24 \, \stagrat^{2k}}{B^4\alpha^{2k}\r^{2k}} \\
        &:= P.
    \end{align*}
    
    When we shift and scale each point $s_0 \in \calS_0$ in finite arithmetic,
    \begin{align*}
        |\fl(r + w + \ax R s_0) - r + w + \ax R s_0| 
        &\le \frac{3\mach}{1 - 3\mach}|r + w + \ax R s_0| \\ 
        &\le 4\mach\left(|r| + \varepsilon + (1 + \varepsilon)        1.001\r\cp\K^{1/k}\pot(H)\right) \\
        &\le 4\mach \left(\varepsilon + 2(1 + \varepsilon)C\scale\right) \\
        &\le 0.1\varepsilon \cdot 1.998\, \r\alpha \K\wacc \\
        &\le 0.1 \varepsilon\ax R
    \end{align*}
    from our assumption on $\mach$, which means that the computed $\calS$ still contains a $\varepsilon \ax R$-net of $\disk(r,\ax R)$. We will assume for simplicity that one can perform the intersection in the final line of $\exc$ while preserving the property that $\calS$ is a maximal $\varepsilon$-net of $\disk(r,\ax R))$ ---this can be achieved, e.g., by intersecting with a slightly larger set and projecting all points outside $\disk(r,\ax R))$ to this latter set. Since $\calS$ is a maximal $\varepsilon$-net of $\disk(r,\ax R))$, it has size at most $9/\varepsilon^2$, and we may recycle a calculation from \cite{banks2021global},
    \begin{align*}
        \max_{s \in \calS}\tau_{(z-s)^k}^{-2k}(H) 
        &\ge \frac{P}{9\K^2\varepsilon^{2k-2}\ax R^{2k}}
        \ge \frac{1}{(1 - \gamma)^{2k}\pot^{2k}(H)}
    \end{align*}
    provided that $\varepsilon$ is no larger than
    \begin{align*}
        \left(\frac{P(1 - \gamma)^{2k}\pot^{2k}(H)}{9\K^2 \ax R^{2k}}\right)^{\frac{1}{2k-2}}
        &\ge \left(\frac{0.24 \stagrat^{2k}(1-\gamma)^{2k}}{\K^6\cp^{2k}\r^{2k}\cdot 9 \cdot 2.001^2 \r^{2k}\cp^{2k} \K^2}\right)^{\frac{1}{2k-2}} \\
        &\ge \left(\frac{\stagrat(1 - \gamma)}{(13 \K^4)^{1/k} \cp^2\r^2}\right)^{\frac{k}{k-1}},
    \end{align*}
    which is the expression appearing in line \ref{line:exc2} of $\exc$.
    
    On the other hand, after the random translation, one can quickly show that every $s \in \calS$ is forward stable with high probability. Because the net is maximal (meaning that no two of the points in it are within $\varepsilon\ax R$ of one another) each eigenvalue $\lambda \in \Spec H$ lies within distance $\varepsilon\ax R$ of at most three points in the net, so the probability that $\dist(\lambda, \calS) < \eta$ after the random translation is at most $3 \eta^2/\varepsilon^2\ax R^2$. Thus the probability that $\dist(\Spec(H), \calS) < \eta$ after the random translation is at most $3n\eta^2/\varepsilon^2\ax R^2$. To ensure that this is smaller than the failure probability $\phi$, we can safely set
    $$
        \eta = \frac{\varepsilon \ax R \sqrt{\phi}}{\sqrt{3n}} \ge \left(\frac{\stagrat(1 - \gamma)}{(13 \K^4)^{1/k} \cp^2\r^2}\right)^{\frac{k}{k-1}} \cdot \frac{1.998 \r \cp \K^{1/k}\wacc\sqrt{\phi}}{\sqrt{3n}}.
    $$
    In the event that the shifts are all forward stable, the definition of $\mach_{\exc}$ means that we can invoke Lemma \ref{lem:multi-iqr-potential-apx}: either some subdiagonal of $\iqr(H,(z-s)^k)$ is smaller than $\wacc$, or $\iqr(H,(z-s)^k)$ satisfies
    $$
        \pot(\iqr(H,(z-s)^k)) < 1.0011\pot(\exactqr(H,(z-s)^k)) \le 1.0011\tau_{(z-s)^k}(H) \le 1.0011(1-\gamma)\pot(H).
    $$
    
    One practical choice of the of the initial $.99\varepsilon$-net of $\disk(0,(1 + \varepsilon))$ is to take an equilateral triangular lattice with spacing $\sqrt{3}\varepsilon$ and intersect it with $\disk(0,(1 + 1.99\varepsilon))$; since this lattice gives an optimal planar sphere packing, it is the optimal choice of net as $\varepsilon \to 0$. Other choices may be more desirable when $\varepsilon$ is large. Adapting an argument of \cite[Lemma 2.6]{armentano2015randomized} (which in turn uses \cite[Theorem 3,  p327]{blum1998complexity}) one can show that with this choice of $\calS_0$,
    \begin{align}
        |\calS| \le |\calS_0| &\le \frac{2\pi}{3\sqrt 3}\left(1.99 + \frac{1}{0.99\varepsilon}\right)^2 + \frac{4\sqrt 2}{\sqrt 3}\left(1.99 + \frac{1}{0.99\varepsilon}\right) + 1 \nonumber \\
        &:= S(\varepsilon) \label{eq:S-def}
    \end{align}
    We will see that every time $\exc$ is called in the course of the full algorithm $\shqr$, the same $\varepsilon$ is used, depending only on the global data. Thus the original net of $\disk(0,1 + \varepsilon)$ need only be computed once, and can be regarded a fixed overhead cost of the algorithm. Given the original net, computing $\calS$ costs one arithmetic operation to add $r + w$, followed by $|\calS_0|$ each to scale and shift by $r + w$. Add to this the operations to compute $\ax\pot(H)$ and $\ax R$, and the cost of obtaining the single random sample, and we get a total of
    $$
        2|\calS_0| + \Croot k \log(k \log \tfrac{1}{1 - 0.999^{1/k}}) + O(1)
    $$
    arithmetic operations. Bounding $|\calS_0| \le S(\varepsilon)$ yields the assertion of the lemma.
\end{proof}

\paragraph{Analysis of $\shkb$.} We now specify and analyze the complete shifting strategy $\Sh_{k,
K}$. \\

\noindent 
\begin{boxedminipage}{\textwidth}
$$\Sh_{k, \K}$$
    \textbf{Input:} Hessenberg $H$, $\r$-optimal Ritz values $\calR$ of $H$, working accuracy $\wacc$, failure probability tolerance $\phi$. \\
    \textbf{Global Data:} Condition number bound $\K$, decoupling rate $\gamma$, norm bound $\scale$, optimality parameter $\r$, promising parameter $\cp$ \\
    \textbf{Output:} Hessenberg $\next{H}$.\\
    \textbf{Requires:} $H$ is $\wacc$-unreduced and $\kappa_V(H) \le \K$\\
    \textbf{Ensures:} With probability at least $1 - \phi$, either $\next{H}$ is $\wacc$-decoupled or $ \pot(\next{H})\le 1.002(1-\gamma)\psi_k(H)$
    \begin{enumerate}
        \item \label{line:sh1} $r \gets \find(H,\calR)$
        \item \label{line:sh2} \textbf{If} $\comptau{k}(H,(z-r)^k) < (1 - \gamma)^k\pot^k(H)$, output $\next H = \iqr(H, (z-r)^k)$.
        \item \textbf{Else}, $\calS \gets \exc(H,r,\wacc,0.999(1-\gamma),\phi)$.
        \item \label{line:sh4} \textbf{For} each $s \in \calS$, \textbf{if} $\pot(\iqr(H,(z -s)^k)) < 1.002(1 - \gamma)\pot(H)$ or some subdiagonal of $\iqr(H,(z-s)^k)$ is smaller than $\wacc$, output $\next H = \exactqr(H,(z-s)^k)$
    \end{enumerate}
\end{boxedminipage}

\begin{lemma}[Guarantees for $\shkb$]
    \label{lem:sh-guarantees}
    Assume that $|r| + 1.001\r\cp\K^{1/k}\pot(H) \le C\|H\|$ and
    \begin{align}
        \mach &\le \mach_{\Sh}(n,k,C,\scale,\K,\dist(\calR,\Spec H),\r,\wacc,\phi,\gamma,\cp) \\
        &:= \min\Big\{\mach_{\find}(n,k,C,\scale,\K,\dist(\calR,\Spec H)),  \nonumber \\
            &\qquad\qquad\qquad \mach_{\exc}\left(n,k,C,\scale,\K,\r,\wacc,\phi,\gamma,0.999(1-\gamma),\cp\right), \nonumber \\
            &\qquad\qquad\qquad \mach_{\ref{lem:multi-iqr-potential-apx}}(n,k,C,\scale,\K,\dist(\calR,\Spec H),\wacc)
        \Big\} \\
        &= 2^{-O\left(k\log\frac{n\scale\K\r\cp}{(1 - \gamma)\wacc\phi\dist(\calR,\Spec H)}\right)} \nonumber
    \end{align}
    Then, $\Sh_{k,\K}$ satisfies its guarantees, and runs in at most 
    \begin{align*}
        T_{\Sh}(n,k,\gamma,\K,\cp,\r) 
        &:= T_{\find}(n,k) + T_{\comptau{}}(n,k) + T_{\exc}(n,k,0.999(1 - \gamma),(1-\gamma),\K,\cp,\r) 
        \\
        &\qquad\qquad + S\left(\left(\frac{0.999(1 - \gamma)^2}{(13 \K^4)^{1/k}\cp^2\r^2}\right)^{\frac{k}{k-1}}\right)\left(T_{\iqr}(n,k) + T_{\psi}(n,k)\right) \\
        &=O\left(kn^2 \K^{\frac{8}{k-1}}\left(\frac{\cp\r}{(1 - \gamma)}\right)^{\frac{4k}{k-1}}\right)
    \end{align*}
    arithmetic operations.
\end{lemma}

\begin{proof}
    The definition of $\mach_{\Sh}$ ensures that $\exc$ and $\find$ (and therefore $\comptau{}$) satisfy their guarantees when called in the course of $\Sh$; the analysis of $\Sh$ is accordingly straightforward. In line \ref{line:sh1}, $\find$ produces an $\cp$-promising, $\r$-approximate Ritz value $r$ for $\cp = (1.01\K)^{\frac{4\log k}{k}}$ as in Table \ref{table:qrii-global-data}; in line \ref{line:sh2} --- because every subdiagonal of $H$ is assumed larger than $\wacc$ --- we know from definition of $\mach_{\Sh}$ and Lemma \ref{lem:multi-iqr-potential-apx} that if $\comptau{k}(H,(z-r)^k) \le (1-\gamma)^k \pot^k(H)$, then
    \begin{align*}
    \pot(\iqr(H,(z - r)^k)) 
        &\le 1.0011\pot(\exactqr(H,(z-r)^k)) \\
        &\le 1.0011\tau_{(z-r)^k}(H) \\
        &\le 1.0011\cdot \left(1.001 \comptau{k}(H,(z-r)^k)\right)^{1/k} \\
        &\le 1.002(1-\gamma)\pot(H).
    \end{align*}
    On the other hand, if $\comptau{k}(H,(z-r)^k) > (1-\gamma)^k \pot^k(H)$ in line \ref{line:sh2}, then using the guarantees for $\comptau{k}$,
    $$
        \tau_{(z-r)^k}^k(H) > 0.999\comptau{k}(H,(z-r)^k) \ge 0.999(1-\gamma)^k\pot(H).
    $$
    Finally, $\exc$ satisfies its guarantees from Lemma \ref{lem:exc-guarantees} when called with $\cp = (1.01 \K)^{\frac{4\log k}{k}}$ and $\stagrat = 0.999^{1/k}(1-\gamma)$. Thus with probability at least $1 - \phi$ at least one exceptional shift $s \in \calS$ satisfies either decoupling (some subdiagonal smaller than $\wacc$) or potential reduction ($\pot(\iqr(H,(z-s)^k)) \le 1.0011(1 - \gamma)\pot(H) \le 1.002(1-\gamma)\pot(H)$).
    
    For the arithmetic operations, $\Sh_{k,\K}$ requires one call to $\find$, one to $\comptau{k}$, one to $\exc$ with stagnation ratio $\stagrat = 0.999(1-\gamma)$, and finally $|\calS|$ calls to degree-$k$ $\iqr$. We can bound $|\calS| \le S(\varepsilon)$, where $\varepsilon$ is defined in the course of $\exc$ with stagnation ratio parameter $\stagrat = 0.999(1 - \gamma)$, and $S(\cdot)$ is defined in \eqref{eq:S-def}. Since and checking every shift in $\calS$ for potential reduction dominates the arithmetic operations, we get that
    $$
        T_{\Sh}(n,k,\K,\gamma,\cp,\r) = O\left(kn^2 \cdot \K^{\frac{8}{k-1}}\left(\frac{\cp\r}{(1-\gamma)}\right)^{\frac{4k}{k-1}}\right).
    $$
\end{proof}
\section{Finite Arithmetic Analysis of $\shqr$}

\subsection{Preservation of $\gap$ and $\kappa_V$}\label{sec:preserve}

\begin{lemma} 
\label{lem:gap-kappaV-perturbation}
Suppose $M$ has distinct eigenvalues. Then for any $E$ satisfying 
\begin{equation}\label{eqn:pertsize}
    \|E\| \le \frac{\gap(M)}{8n^2\cdot \kappa_V^3(M)}
\end{equation}
we have
\begin{equation}\label{eqn:pertgap}
    \gap(M+E)\ge \gap(M)-2\kappa_V(M)\|E\|
\end{equation}
and
\begin{equation}\label{eqn:pertkappav}
    \kappa_V(M+E)\le \kappa_V(M)+ 6n^2\frac{\kappa_V^3(M)}{\gap(M)}\|E\|.
\end{equation}

\end{lemma}
\begin{proof}
The assertion in \eqref{eqn:pertgap} is an immediate consequence of the Bauer-Fike theorem. For \eqref{eqn:pertkappav}, let $V$ be scaled so that $\|V\| = \|V^{-1}\| = \kappa_V(A)$, with (not necessarily unit) columns $v_1,...,v_n$ satisfying $Mv_i = \lambda_i v_i$ for each $i \in [n]$. It follows from \cite[Proposition 1.1]{banks2020pseudospectral} that whenever $\|E\| \le \frac{\gap(M)}{8\kappa_V(M)}$, there exists a matrix $V'$ with columns $v_1',...,v_n'$ diagonalizing $M' := M + E$, such that
$$
    \|v_i - v'_i\| \le 2n\frac{\kappa_V(M)}{\gap(M)}\|E\|\|v_i\|,
$$
which implies
$$
    \|V - V'\| \le \|V - V'\|_F \le 2n^{3/2}\frac{\kappa_V(M)}{\gap(M)}\|E\|\|V\|_F \le 2n^2\frac{\kappa_V(M)}{\gap(M)}\|V\|.
$$

It is standard that each singular value of $V'$ satisfies $|\sigma_i(V') - \sigma_i(V)| \le \|V - V'\|$, so using $\|V\| = \|V^{-1}\| = \sqrt{\kappa_V(M)}$, we have
\begin{align*}
    \kappa_V(M') &\le \|V'\|\|(V')^{-1}\| \\
    &\le \frac{\|V\| + \|V - V'\|}{\|V^{-1}\|^{-1} - \|V - V'\|} \\
    &\le \kappa_V(M)\frac{1 + 2n^2\frac{\kappa_V(M)}{\gap(M)}\|E\|}{1 - 2n^2\frac{\kappa_V^2(M)}{\gap(M)}\|E\|} \\
    &\le \kappa_V(M) + \frac{8}{3}n^2(1 + \kappa_V(M))\frac{\kappa_V^2(M)}{\gap(M)}\|E\|,
\end{align*}
where in the final line we have used \eqref{eqn:pertsize} to argue that $2n^2\frac{\kappa_V^2(M)}{\gap(M)}\|E\| \le 1/4$, and convexity of the function $f(x) = \frac{1 + x/\kappa_V(M)}{1 - x}$ to bound by the linear interpolation between $x = 0$ and $x = 1/4$. The advertised bound then follows from applying $\kappa_V(M) \ge 1$ and bounding $16/3 \le 6$.
\end{proof}

\begin{lemma}
    \label{lem:kappav-deflation}
    If $M$ is block upper triangular and $M'$ is a diagonal block, then $\kappa_V(M') \le \kappa_V(M)$ and $\gap(M') \ge \gap(M)$.
\end{lemma}

\begin{proof}
    The gap assertion is immediate since $\Spec M' \subset \Spec M$. For $\kappa_V$, assume without loss of generality that $M$ is diagonalizable (otherwise the inequality is trivial) and
    $$
        M = \begin{pmatrix} M' & \ast \\ 0 & \ast \end{pmatrix}.
    $$
    We claim that every $V$ diagonalizing $M$ is of the form
    $$
        V = \begin{pmatrix} V' & \ast \\ 0 & \ast \end{pmatrix},
    $$
    where $V'$ diagonalizes $M'$. To see this, if $MV = VD$, then block upper triangularity gives $M'V' = V'D'$ for $D'$ the upper left block of $D$. Moreover, $V$ invertible implies $V'$ is as well, and quantitatively $\|V'\|\|(V')^{-1}\| \le \|V\|\|V^{-1}\|$. Choosing $V$ so that $\kappa_V(M) = \|V\|\|V^{-1}\|$, we have
    $$
        \kappa_V(M') \le \|(V')\|\|(V')^{-1}\| \le \|V\|\|V^{-1}\| = \kappa_V(M).
    $$
\end{proof}

\subsection{The Full Algorithm}
\label{sec:fullalg}
We are now ready to analyze, in finite arithmetic, how the shifting strategy $\Sh_{k,\K}$ introduced in \cite{banks2021global} can be used to approximately find all eigenvalues of a Hessenberg matrix $H$. One simple subroutine is required in addition to the ones described in the preceding sections: $\deflate(H,\wacc,k)$ takes as input a Hessenberg matrix $H$, deletes any of the bottom $k-1$ subdiagonal entries smaller than $\wacc$, and outputs the resulting diagonal blocks $H_1,H_2,...$. It runs in  $T_{\deflate}(H,\wacc,k) = k$ arithmetic operations. \\

\noindent
\begin{boxedminipage}{\textwidth}
$$\shqr$$
\textbf{Input:} Hessenberg matrix $H$, accuracy $\delta$, failure probability tolerance $\phi$ \\
\textbf{Global Data:} Eigenvector condition number bound $\K$, eigenvalue gap bound $\Gamma$, matrix norm bound $\scale$, original matrix dimension $n$\\
\textbf{Requires:} $\scale \ge 2\|H\|$, $\K \ge 2\kappa_V(H)$, $\Gamma \le \gap(H)/2$, $\acc \le \scale$ \\
\textbf{Output:} A multiset $\Lambda \subset \bC$\\
\textbf{Ensures:} With probability at least $1 - \phi$, $\Lambda$ are the eigenvalues of some $\ax H$ with $\|\ax H - H\| \le \acc$\\
\begin{enumerate}
    \item \label{line:shqr-set-parameters} 
    $\wacc \gets \frac{1}{4n} \min\left\{\acc,\frac{\Gamma}{8n^2 \K^2}\right\}$, $\varphi \gets \frac{\phi}{3n^2}\frac{\log\frac{1} {1.002(1-\gamma)}}{\log \frac{\scale}{\wacc}}$
    \item \label{line:shqr-check-dim} \textbf{If} $\text{dim}(H) \le k$, $\Lambda \gets  \smalleig(H,\acc,\phi)$, output $\Lambda$ and halt
    \item \label{line:shqr-else} \textbf{Else} $\Lambda \gets \emptyset$ and
    \begin{enumerate}
        \item \label{line:shqr-while-not-decoupled} \textbf{While} $\min_{n-k+1\leq i \leq n} H_{i,i-1} > \wacc$
        \begin{enumerate}
            \item \label{line:shqr-ROD} $[\calR,\next H,\decouple] = \regularize(H,\wacc,\varphi)$
            \item \label{line:shqr-decouple-true} \textbf{If} $\decouple = \true$, $H \gets \next H$ and end while
            \item \label{line:shqr-decouple-false} \textbf{Else if} $\decouple = \false$, $H \gets \Sh_{k,B}(H,\calR,\wacc,\varphi)$ 
        \end{enumerate}
        \item \label{line:shqr-deflate} $[H_1,H_2,...H_\ell] = \deflate(H,\wacc)$
        \item \label{line:shqr-recurse} 
        \textbf{For} each $j \in [\ell]$
        \begin{enumerate}
            \item \textbf{If} $\dim(H_j) \le k$, $\Lambda \gets \Lambda \sqcup \smalleig(H_j,\acc/n,\phi/3n)$
            \item \textbf{Else} repeat lines \ref{line:shqr-while-not-decoupled}-\ref{line:shqr-recurse} on $H_j$
        \end{enumerate}  
    \end{enumerate}
\end{enumerate}
\end{boxedminipage}

\begin{theorem}[Guarantees for $\shqr$]
    \label{thm:shqr-guarantees}
    Let $k$, $\r$, $\cp$, and $\gamma$ be set in terms of $\K$ as in \eqref{eqn:setk}, $N_{\mathsf{dec}}$ be defined in \eqref{eq:ndec}, and $\wacc$ and $\varphi$ be defined in line \ref{line:shqr-set-parameters} of $\shqr$. Assuming
    \begin{align}
        \mach 
        &\le \mach_{\shqr}(n,k,\scale,\K,\acc) \nonumber \\
        &:= \min\left\{
        \frac{\wacc}{4.5 k N_{\mathsf{dec}} \cdot n \muqr(n) \scale}, \mach_{\regularize}\left(n,k,\scale,\K,\r,\wacc,\varphi\right)\right., \nonumber \\
        &\qquad\qquad\qquad \left. \mach_{\Sh}\left(n,k,3,\scale,\K,\frac{\wacc^2\sqrt{\varphi}}{32\cdot 101 \cdot \scale\sqrt{2k}},\r,\wacc,\varphi,\gamma,\cp\right)\right\} \label{eq:mach-requirement-shqr} \\
        &= 2^{-O\left(k\log\frac{n\scale\K}{\acc\Gamma\phi}\right)}, \nonumber
     \end{align}
    $\shqr$ satisfies its guarantees and runs in at most
    \begin{align}
        T_{\shqr}(n,k,\acc,\K,\scale,\gamma) &\le n\Big( T_{\regularize}(n,k,\true) \nonumber \\
        &\qquad + N_{\mathsf{dec}}\Big(T_{\regularize}(n,k,\false) + T_{\Sh}(n,k,\gamma,\K,\cp,\r)\Big) \label{eq:t-shqr} \\
        &\qquad + T_{\deflate}(k) \Big) \nonumber \\
        &= O\left(\left(\log\frac{n\K\scale}{\acc\Gamma}k\log k + k^2\right)n^3\right) \nonumber
    \end{align}
    arithmetic operations, plus $O(n\log \frac{n\K\scale}{\acc\Gamma})$ calls to $\smalleig$ with accuracy $\Omega(\frac{\Gamma^2}{n^4\K^4 \scale})$ and failure probability tolerance $\Omega(\frac{\phi}{n^2\log \frac{n\K\scale}{\acc\Gamma}})$.
\end{theorem}

In the above result, we assume access to an upper bound $\scale \ge 2\|H\|$ and show that $\shqr$ can approximate the eigenvalues of $H$ with (absolute) backward error $\acc$, whereas in our main Theorem \ref{thm:main}, we ask for (relative) backward error $\acc$. To prove Theorem \ref{thm:main} from Theorem \ref{thm:shqr-guarantees}, we need only compute an upper bound $\scale \ge 2\|H\|$ and call $\shqr$ with accuracy $\acc/\scale$. Such a bound can be computed (for instance) using random vectors or, at the cost of a factor of $\sqrt n$, by taking the Frobenius norm of $H$. In either case, the arithmetic cost and precision are dominated by the requirements for $\shqr$ itself.

\begin{proof}[Proof of Theorem \ref{thm:shqr-guarantees}]
    At a high level, $\shqr$ is given an input matrix $H$,  $\wacc$-decouples $H$ to a unitarily similar matrix $\next H$ via a sequence of applications of $\regularize + \Sh_{k,\K}$, deflates $\next H$ to a block upper triangular matrix with diagonal blocks $H_1,...,H_\ell$, then repeats this process on each block $H_j$ with dimension larger than $k\times k$. Since the effect of $\regularize$ and $\Sh_{k,\K}$ on any input matrix $H'$ is approximately a unitary conjugation, it will be fruitful for the analysis to regard each of the blocks $H_1,...,H_\ell$ as embedded in the original matrix, and promote the approximate unitary conjugation actions of the subroutines on each block to unitary conjugations of the full matrix. The same goes once each of $H_1,...,H_\ell$ is decoupled and deflated and we pass to further submatrices of each one. Importantly, this viewpoint is necessary \textit{only} for the analysis: the algorithm need not actually manipulate the entries outside the blocks $H_1,...,H_\ell$. In this picture, the end point of the algorithm is a matrix of the form
    \begin{equation}
        \label{eq:shqr-endpoint}
        \begin{pmatrix} L_1 & \ast & \ast \\ & L_2 & \ast \\ & & \ddots \end{pmatrix},
    \end{equation}
    where $L_1,L_2,...$ are all $k\times k$ or smaller matrices on which $\smalleig$ can be called directly, and the $\ast$ entries are unknown and irrelevant to the algorithm. By the guarantees on $\smalleig$ (and the fact that $\beta$-forward approximation of eigenvalues implies $\beta$-backward approximation), the output of the algorithm is thus
    $$
        \bigsqcup_j \smalleig(L_j,\wacc,\varphi) = \bigsqcup_j \Spec \ax L_j = \Spec  \begin{pmatrix} \ax L_1 & \ast & \ast \\ & \ax L_2 & \ast \\ & & \ddots \end{pmatrix}
    $$
    where $\ax L_1,\ax L_2,...$ are some matrices satisfying $\|L_j - \ax L_j\| \le \acc/n$, and the remaining entries are identical to those in \eqref{eq:shqr-endpoint}. Our goal in the proof will thus be to show that for some unitary $\ax Q$,
    $$
        \left\| \begin{pmatrix} L_1 & \ast & \ast \\ & L_2 & \ast \\ & & \ddots \end{pmatrix} - \ax Q^\ast H \ax Q \right\| \le \acc - \acc/n,
    $$
    where the left hand matrix is a block upper triangular matrix with the blocks $L_1, L_2,...$ on the diagonal. This will in turn imply that
    \begin{align*}
        \left\| \begin{pmatrix} \ax L_1 & \ast & \ast \\ & \ax L_2 & \ast \\ & & \ddots \end{pmatrix} - \ax Q^\ast H \ax Q \right\| 
        &\le \left\| \begin{pmatrix} L_1 - \ax L_1 & \ast & \ast \\ & L_2 - \ax L_2 & \ast \\ & & \ddots \end{pmatrix} \right\| + \acc - \acc/n \\
        &\le \max_i \|L_i - \ax L_i\| + \acc - \acc/n \le \acc,
    \end{align*}
    as desired.
    
    We begin by analyzing the while loop in line \ref{line:shqr-while-not-decoupled}.
    
    \begin{lemma}
        \label{lem:while-guarantees}
        Assume that at during the execution of $\shqr$, the while loop in line \ref{line:shqr-while-not-decoupled} is initialized with a matrix $H'$ satisfying $\|H'\| \le (1 - 1/2n)\scale$, $\kappa_V(H') \le (1 - 1/2n)\K$, and $\gap(H') \ge (1 + 1/2n)\Gamma$. Let
        \begin{equation}
        \label{eq:ndec}
            N_{\mathsf{dec}} := \frac{\log \frac{\scale}{\wacc}}{\log \frac{1}{1.002(1-\gamma)}}.
        \end{equation}
        If
        $$
            \mach \le \mach_{\shqr}(n,k,\scale,\K,\acc),
        $$
        then the loop terminates in at most $N_{\mathsf{dec}}$ iterations, having produced a $\wacc$-decoupled matrix $\next{H'}$ at most $\wacc$-far from a unitary conjugate of $H'$.
    \end{lemma}
    
    \begin{proof}[Proof of Lemma]
        Let us write $H''$ for the matrix produced by several runs through lines \ref{line:shqr-ROD}-\ref{line:shqr-decouple-false}, after the while loop has been initialized with $H'$, and assume that all prior calls to $\regularize$ or $\Sh_{k,\K}$ during the loop have satisfied their guarantees, and moreover that all prior shifts have had modulus at most $4.5\|H'\|$ in the complex plane. We will show inductively that this last condition holds throught the while loop.
        
        Because the prior calls to $\regularize$ and $\Sh_{k,\K}$ satisfy their guarantees, each previous run through lines \ref{line:shqr-ROD}-\ref{line:shqr-decouple-false} has either effected immediate decoupling or potential reduction by a multiplicative $1.002(1 - \gamma)$. Since $\wacc \le \pot(H') \le \|H'\| \le \scale$, there can have been at most $N_{\mathsf{dec}}$ runs through lines \ref{line:shqr-ROD}-\ref{line:shqr-decouple-false} so far, each of which we can regard as an IQR step of degree $k$, meaning that we can think of $H''$ as being produced from $H'$ by a \textit{single} IQR step of degree $kN_{\mathsf{dec}}$.\footnote{This is because we have simply defined a higher degree $\iqr$ step as a composition of many degree $1$ $\iqr$ steps.} Thus by Lemma \ref{lem:iqr-multi-backward-guarantees}, our inductive assumption on the prior shifts, and the hypothesis on $\mach$, the distance from $H''$ to a unitary conjugate of $H'$ is at most $4.5\|H\|kN_{\mathsf{dec}}\muqr(n)\mach \le \wacc$. If $H''$ is $\wacc$-decoupled, then the while loop terminates, and the proof is complete. 
        
        Otherwise $H''$ is not $\wacc$-decoupled. By the definition of $\wacc$ and the fact that $\wacc \le \acc/2n \le \scale/2n$, we can apply Lemma \ref{lem:gap-kappaV-perturbation} to find
        \begin{align*}
            \|H''\| &\le \|H'\| + \wacc \le (1 - 1/2n)\scale + \scale/2n \le \scale \\
            \kappa_V(H'') &\le \kappa_V(H') + 6 n^2\frac{\kappa_V^3(H')}{\gap(H')} \wacc \le (1 - 1/2n)\K + \K/2n \le \K \\
            \gap(H'') &\ge \gap(H') - 2\kappa_V(H')\wacc \ge (1 + 1/2n)\Gamma - \Gamma/2n \ge \Gamma,
        \end{align*}
        and we furthermore have $2\wacc^2/\scale \le 2\wacc \le \Gamma \le \gap(H'')$ by the above and the definition of $\wacc$. This means  $\regularize(H'',\wacc,\varphi)$ meets its requirements, and from our assumption on $\mach$ we can apply Lemma \ref{lem:regularize-guarantees} to conclude that it satisfies its guarantees. If this call to $\regularize$ outputs $\decouple = \true$, then the matrix it outputs is indeed decoupled and the while loop terminates. 
        
        If on the other hand $\decouple = \false$, then $\regularize$ outputs $H''$ and $\r$-approximate Ritz values $\calR$ contained in in a disk of radius $1.1\|H''\|$, and $\regularize$ guarantees
        $$
            \dist(\calR,H'') \ge  \frac{\wacc^2\sqrt{\varphi}}{32\cdot 101 \cdot \scale\sqrt{2k}}. 
        $$ 
        The bound on $\kappa_V(H'')$ in the previous paragraph ensures that the requirements of $\Sh_{k,\K}(H,\calR,\wacc,\varphi)$ have been met, and the parameter settings in \eqref{eqn:setk}-\eqref{eqn:settheta} give us
        \begin{align*}
            1.001\r\cp\K^{1/k}\pot(H'')
            &= 1.001\frac{1.01}{0.998^{1/k}}(2\K^4)^{1/2k}(1.01 \K)^{\frac{4\log k}{k}}\K^{1/k}\pot(H'') \\
            &= 1.04 \cdot 2^{1/2k}\K^{\frac{4\log k + 3}{k}} \pot(H'') \\
            &\le 1.04 \cdot\sqrt{2^{\frac{2}{k-1}}\K^{\frac{8\log k + 11}{k-1}}} \|H''\| \\
            &\le 1.04 \sqrt{3} \|H''\| \\
            &\le 1.9 \|H''\|,
        \end{align*} 
        so every exceptional shift has modulus at most $3\|H''\|$ in the complex plane. Our assumption on $\mach$ lets us invoke Lemma \ref{lem:sh-guarantees} to conclude that $\Sh_{k,\K}$ achieves potential reduction by a multiplicative factor of $1.002(1- \gamma)$. Moreover, the shifts executed by $\regularize$ and $\Sh$ in the above run through the while loop had modulus at most
        $$
            3\|H''\| \le 3\|H'\|(1 + 4.5 kN_{\mathsf{dec}}\muqr(n)\mach) \le 3\|H'\|\cdot(1 + \wacc/\scale) \le 4.5\|H'\|,
        $$
        again since $\wacc \le \acc/2n \le \scale$. 
        
        The proof above ensures that for each of its first $N_{\mathsf{dec}}$ iterations, the while loop either produces decoupling or potential reduction by a multiplicative $1.002(1 - \gamma)$, and our earlier discussion implies that it therefore terminates after after at most $N_{\mathsf{dec}}$ iterations. When it does, the proof above additionally tells us that the final matrix $\next{H'}$ is at most $\wacc$-far from a unitary conjugate of $H'$, as desired.
    \end{proof}
    
    We next check that each time the while loop begins in the course of $\shqr$, the hypotheses of Lemma \ref{lem:while-guarantees} are satisfied. This is immediate the first time the loop begins, where the requirements of $\shqr$ give $\|H\| \le \scale/2$, $\kappa_V(H) \le \K/2$, and $\gap(H) \ge 2\Gamma$. If $H'$ is a matrix passed to the while loop, and each of the while loops in its production has satisfied the conclusion of Lemma \ref{lem:while-guarantees}, then $H'$ is the result of at most $n-1$ of decouplings-and-deflations, each of which caused the norm, eigenvector condition number, and gap to deteriorate by at worst an additive $2\wacc$. Thus, finally using the full force of the $1/4n$ factor in the definition of $\wacc$,
    \begin{align*}
        \|H'\| &\le \|H\| + 2(n-1)\wacc \le (1 - 1/2n)\scale \\
        \kappa_V(H') &\le \kappa_V(H) + 6 n^2\frac{\kappa_V^3(H)}{\gap(H)} \cdot 2(n-1)\wacc \le (1 - 1/2n)\K \\
        \gap(H') &\ge \gap(H) - 2\kappa_V(H)\cdot 2(n-1)\wacc \ge (1 + 1/2n)\Gamma
    \end{align*}
    by the definition of $\wacc$. 
    
    This ensures that \textit{every} execution of the while loop throughout $\shqr$ satisfies the conclusion of Lemma \ref{lem:while-guarantees}, which means that the set of `base case' matrices $L_1,L_2,...$ are produced by a tree of alternating decouplings and deflations with depth at most $n-1$, and moreover that
    $$
        \left\| \begin{pmatrix} L_1 & \ast & \ast \\ & L_2 & \ast \\ & & \ddots \end{pmatrix} - \ax Q^\ast H \ax Q \right\| \le 2(n-1)\wacc \le \acc - \acc/n,
    $$
    for some unitary $\ax Q$, as we had set out to show. \\
    
    \noindent\textit{Failure Probability.\,\,} We have already shown that $\regularize$ and $\Sh_{k,\K}$ satisfy their guarantees (including their failure probability) throughout $\shqr$ whenever the hypotheses of Theorem \ref{thm:shqr-guarantees}; these, plus the base calls to $\smalleig$, are the only sources of randomness in the algorithm. There are at most $n^2 \cdot N_{\mathsf{dec}}$ calls each to $\regularize$ and $\Sh_{k,\K}$ over the course of the algorithm, each failing with probability $\varphi$, and at most $n$ calls to $\smalleig$, each failing with probability at most $\phi/3n$. By a union bound and the definition of $\varphi$, the total failure probability is at most $\phi$. \\
    
    \noindent\textit{Arithmetic Operations and Calls to $\smalleig$.\,\,} $\shqr$ recursively runs through line \ref{line:shqr-else} many times in the course of the algorithm; write $T_{\ref{line:shqr-else}}(m,k,\acc,\K,\scale,\Gamma)$ for the arithmetic operations required to execute this line on some matrix of size $m\times m$ during the algorithm, with the convention that this quantity is zero when $m \le k$. Then we have
    \begin{align*}
        T_{\shqr}(n,k,\acc,\K,\scale,\Gamma) 
        &= T_{\ref{line:shqr-else}}(n,k,\acc,\K,\scale,\Gamma) \\
        &\le T_{\regularize}(n,k,\true) \\
        &\qquad + N_{\mathsf{dec}}\Big(T_{\regularize}(n,k,\false) + T_{\Sh}(n,k,\acc,\K,\scale,\Gamma)\Big) \\
        &\qquad + T_{\deflate}(k) + \max_{\sum_i n_i = n}\sum_i T_{\ref{line:shqr-else}}(n_i,k,\acc,\K,\scale,\Gamma).
    \end{align*}
    Since each of the expressions $T_{\square}(\cdot)$ is a polynomial of degree at most two in $n$, the maximum in the third line can be bounded by $T_{\ref{line:shqr-else}}(n-1,k,\acc,\K,\scale,\Gamma)$. Losing only a bit in the constant, we can bound as
    \begin{align*}
        T_{\shqr}(n,k,\acc,\K,\scale,\gamma) &\le n\Big( T_{\regularize}(n,k,\true) \\
        &\qquad + N_{\mathsf{dec}}\Big(T_{\regularize}(n,k,\false) + T_{\Sh}(n,k,\acc,\K,\scale,\Gamma)\Big) \\
        &\qquad + T_{\deflate}(k) \Big) \\
        &= O\left(\left(\log\frac{n\K\scale}{\acc\Gamma}k\log k + k^2\right)n^3\right).
    \end{align*}
    In addition, $\shqr$ requires at most $ O(n\log \tfrac{n\K\scale}{\acc\Gamma})$ calls to $\smalleig$ with accuracy $\Omega(\wacc^2/\scale)$ and failure probability tolerance $\varphi$ in the course of the calls to $\regularize$, plus $O(n)$ `base case' calls with accuracy $\acc/n$ and failure probability tolerance $\phi/3n$; the latter calls to $\smalleig$ are asymptotically dominated by the former. The estimates in the theorem statement come from bounding $\wacc$ and $\varphi$.
\end{proof}

\bibliographystyle{alpha}
\bibliography{ShiftedQR}

\appendix
\section{Deferred proofs from Section \ref{sec:implicitQR}}\label{sec:implicitQRdeferred}
\begin{proof}[Proof of Lemma \ref{lem:iqr-single-guarantees}]
    For the purpose of the analysis, let us define $\ax{H_0} := H - s$ and for each $i = 1,...,n-1$, denote by $\ax{H}_i$ the matrix $\ax{R}$ as it stands at the end of line 2(c) on the the $i$th step of the loop. Additionally, write $G_i$ for the unitary matrix which applies $\giv(X_{1:2,i})$ to the span of $e_i$ and $e_{i+1}$ and is the identity elsewhere. We will show that the unitary $\ax Q := \ax{Q}_{n-1}$ satisfies the guarantees of $\iqr$. We then have
    $$
        \ax{H}_i = G_i^\ast \ax{H}_{i-1} + E_{2,i}, 
    $$
    where $E_{2,i}$ is the structured error matrix which in rows $(i:i+1)$ is equal to
    $$
        \begin{pmatrix} \begin{matrix} E_{2,i,c}\, \\ 0 \end{matrix} \vline & \Large{E_{2,i,b}} \end{pmatrix}
    $$
    and is zero otherwise. From the discussion at the beginning of this appendix, we know that each entry of $E_{2,i,b}$ has size at most $8\|\ax{H}_{i-1}\|\mach$ and similarly that $|E_{2,i,c}| \le 2\|X_{1:2,i}\|\mach \le 8\|\ax{H}_{i-1}\|\mach$. Thus $\|E_{2,i}\| \le 8\sqrt{n}\|\ax{H}_{i-1}\|\mach$, and inductively we have 
    \begin{align*}
        \|\ax{H}_i\| 
        &\le \|\ax{H}_{i-1}\| + \|E_{2,i}\| \\
        &\le \|\ax{H}_{i-1}\|\left(1 + 8\sqrt n\mach\right) \\
        &\le \|\ax{H}_0\|\left(1 + 8\sqrt n\mach\right)^n \\
        &\le \|\ax{H}_0\|\exp\left(8n^{3/2}\mach\right) \\
        &\le 2\|H - s\| \qquad i = 1,...,n-1.
    \end{align*}
    Since $\ax Q$ and every $G_i$ is unitary, this gives
    $$
        \|H - s - \ax{Q}\ax{R}\| = \|\ax{Q}^\ast\ax{H_0} - \ax{R}\| \le \sum_{i \in [n-1]}\|E_{2,i}\| \le 16 n^{3/2} \mach \cdot \|H - s\|.
    $$
    A similar inductive argument applied to line 4 gives that $\|E_{4,i}\| \le 16 \sqrt{n} \mach \cdot \|H - s\|$ for every $i \in [n-1]$, and thus that the $\ax{\next{H}}$ output by $\iqr(H,s)$ satisfies
    \begin{align*}
        \ax{\next{H}} - s
        &= \ax R \ax Q + E_{4,n-1}(G_1\cdots G_{n-2}) + \cdots + E_{4,2}G_1 + E_{4,1} \\
        &= \ax{Q}^\ast(H - s)\ax Q + E_{4,n-1}(G_1\cdots G_{n-2}) + \cdots + E_{4,2}G_1 + E_{4,1} \\
        &\qquad + (G_{n-2}^\ast\cdots G_1^\ast) E_{2,1}\ax{Q} + (G_{n-3}^\ast\cdots G_1^\ast) E_{2,2}\ax{Q} + \cdots + G_1^\ast E_{2,n-1}\ax{Q},
    \end{align*}
    meaning
    $$
        \|\ax{\next{H}} - \ax{Q}^\ast H \ax{Q}\| \le 32 n^{3/2}\mach \cdot \|H - s\|
    $$
    and
    $$
        \|\ax{\next H}\| \le \|H\| + 32 n^{3/2}\|H - s\| \mach,
    $$
    as desired.
    
    In terms of arithmetic operations, it costs $n$ to compute $\ax R$ from $H$ in line 1. In line 2(b), computing $\|X_{1:2,i}\|$ costs $4$, computing $\giv(X_{1:2,i})$ given this norm costs another $2$, zeroing out $\ax{R}_{i+1,i}$ costs $1$, replacing $\ax{R}_{i,i}$ with $\|X_{1:2,i}\|$ costs one, and applying the rotation to $\ax{R}_{i:i+1,i+1:n}$ costs $4(n-i+1)$. We do this for each of $i = 1,2,...n-1$, giving $6(n-1) + 2(n-1) + 2n(n-1)$. In line 4, assuming we have stored each Givens rotation, applying them again requires $2n(n+1)-4$. Finally, in line 5 we pay another $n$ to re-apply the shift. Thus in total we have
    $$
        n + 6(n-1) + 2(n-1) + 2n(n-1) + 2n(n+1) - 4 + n = 4n^2 + 12n - 12 \le 7n^2 \qquad n \ge 2.
    $$
\end{proof}

\begin{proof}[Proof of Lemma \ref{lem:iqr-multi-backward-guarantees}]
    Let $\ax{H}_1 = H$, and for each $\ell \in [m-1]$, let $[\ax{H}_{\ell+1},\ax{R}_\ell] = \iqr(\ax{H}_\ell,r_\ell)$ and $\ax{Q}_\ell$ be as guaranteed by Definition \ref{def:stableiqr}. We have
    $$
        \|\ax{H}_2 - \ax{Q}_1\ast \ax{H_1} \ax{Q}_1\| \le \|\ax{H}_1 - s_1\|\muqr(n)\mach \le (1 + C)\|H\|\muqr(n)\mach,
    $$
    and inductively, assuming that
    $$
        \|\ax{H}_\ell - \ax{Q}_{\ell - 1}^\ast \ax{H}_{\ell - 1}\ax{Q}_{\ell - 1}\| \le (1 + C)\|H\|(\muqr(n)\mach + \cdots + (\muqr(n)\mach)^\ell),
    $$
    we have
    \begin{align*}
        \|\ax{H}_{\ell + 1} - \ax{Q}_\ell^\ast \ax{H}_{\ell}\ax{Q}_\ell\| 
        &\le \|\ax{H}_\ell - s_\ell\|\muqr(n)\mach \\
        &\le \|H\|(1 + (1 + C)(\muqr(n)\mach + \cdots + (\muqr(n)\mach)^\ell) + C)\muqr(n)\mach \\
        &\le (1 + C)\|H\|(\muqr(n)\mach + \cdots + (\muqr(n)\mach)^{\ell+1}).
    \end{align*}
    This gives the first asserted bound, since
    $$
        \|\ax{H} - \ax{Q}^\ast \ax{H}\ax{Q}\| \le \sum_{\ell \in [m-1]} \|\ax{H}_{\ell + 1} - \ax{Q}^\ast_\ell \ax{H}_\ell \ax{Q}_\ell\| \le (1 + C)\|H\| \frac{m\muqr(n)\mach}{1 - \muqr(n)\mach}
    $$
    and $\tfrac{1}{1 - \muqr(n)\mach} \le 4/3 \le 1.4$.
    
    For the second assertion, we will mirror the proof of Lemma \ref{lem:exact-iqr-composition}, using backward stability guarantees on a single $\iqr$ step from Definition \ref{def:stableiqr}. In particular, in view of the definition and the above bound, we can write 
    \begin{align*}
        \ax{H}_\ell - s_\ell &= \ax{Q}_\ell\ax{R}_\ell + E_\ell &  \|E_\ell\| &\le (1 + C)\|H\| \frac{\muqr(n)\mach}{1 - \muqr(n)\mach} \\
        \ax{H}_1\ax{Q}_\ell\cdots\ax{Q}_1 &= \ax{Q}_\ell\cdots\ax{Q}_1\ax{H}_{\ell + 1} + \|\Delta_{\ell+1}\| & \Delta_{\ell + 1} &\le (1 + C)\|H\| \frac{\muqr(n)\mach}{1 - \muqr(n)\mach}
    \end{align*}
    so that
    \begin{align*}
        p(H) = p(\ax{H}_1)
        &= (\ax{H}_1 - s_m)\cdots(\ax{H}_1 - s_1) \\
        &= (\ax{H}_1 - s_m)\cdots(\ax{Q}_1\ax{R}_1 + \ax{Q}_1^\ast E_1) \\
        &= (\ax{H}_1 - s_m) \cdots (\ax{H_1} - s_2)\ax{Q}_1(\ax{R}_1 + \ax{Q}_1^\ast E_1) \\
        &= (\ax{H}_1 - s_m) \cdots \ax{Q}_1(\ax{H}_2 - s_2 + \Delta_2)(\ax{R}_1 + \ax{Q}_1^\ast E_1) \\
        &= (\ax{H}_1 - s_m) \cdots (\ax{H}_1 - s_3)\ax{Q}_1\ax{Q_2}(\ax{R}_2 + \ax{Q}_2^\ast E_2 + \ax{Q}_2^\ast\Delta_2)(\ax{R}_1 + \ax{Q}_1^\ast E_1) \\
        &= \ax{Q}_1 \cdots \ax{Q}_m(\ax{R}_m + \ax{Q}_m^\ast E_m + \ax{Q}_m^\ast\Delta_m)\cdots(\ax{R}_2 + \ax{Q}_2^\ast E_2 + \ax{Q}_2^\ast \Delta_2)(\ax{R}_1 + \ax{Q}_1^\ast E_1)
    \end{align*}
    Thus, using the bounds on $E_\ell$ and $\Delta_\ell$, and the fact that $\|\ax{R}_\ell\| = \|\ax{H}_\ell - s_\ell\| \le \tfrac{(1 + C)\|H\|}{1 - \muqr(n)\mach}$,
    \begin{align*}
        \|p(H) - \ax{Q}_1\cdots\ax{Q}_m\ax{R}_m \cdots \ax{R}_1\| 
        &= \|\ax{R}_m \cdots \ax{R}_1 - (\ax{R}_m + \ax{Q}_m^\ast E_m + \ax{Q}_m^\ast\Delta_m)\cdots(\ax{R}_2 + \ax{Q}_2^\ast E_2 + \ax{Q}_2^\ast \Delta_2)(\ax{R}_1 + \ax{Q}_1^\ast E_1)\| \\
        &\le \prod_{\ell \in [m]}\left(\|\ax{R}_\ell\| + \tfrac{2(1 + C)\|H\|}{1 - \muqr(n)\mach}\right) - \prod_{\ell \in [m]}\|\ax{R}_\ell\| \\
        &\le \left(\tfrac{(1 + C)\|H\|}{1 - \muqr(n)\mach}\right)^m\left((1 + 2\muqr(n)\mach)^m - 1\right) \\
        &\le 4\big(2(1 + C)\|H\|\big)^m\muqr(n)\mach;
    \end{align*}
    in the final line we are using again that $\muqr(n)\mach \le 1/4$ and thus that $((1 + 2\muqr(n)\mach)^m - 1) \le (3/2)^m\muqr(n)\mach/4$, whereas $(1 - \muqr(n)\mach)^{-m} \le (4/3)^m$
\end{proof}

\end{document}